\documentclass[11pt]{amsart}
\usepackage[dvips]{graphicx}
\usepackage{amsmath}
\usepackage{amsthm}

\usepackage{amssymb}
\usepackage{mathrsfs}
\usepackage{latexsym}
\theoremstyle{plain}
\newtheorem{thm}{Theorem}[section]
\newtheorem{lem}[thm]{Lemma}
\newtheorem{cor}[thm]{Corollary}
\newtheorem{prop}[thm]{Proposition}
\theoremstyle{definition}
\newtheorem{rem}[thm]{Remark}
\newtheorem{defn}[thm]{Definition}

 \allowdisplaybreaks 

\numberwithin{equation}{section} 

\newcommand{\R}{\mathbb{R}}
\newcommand{\C}{\mathbb{C}}

\newcommand{\N}{\mathbb{N}}

\newcommand{\pa}{\partial}

\newcommand{\A}{\alpha}

\newcounter{kotaeflg}
\newcommand{\kotae}[1]{
\ifodd \arabic{kotaeflg}
#1
\fi
}

\begin{document}
\title[Global well-posedness for 1D NLS]%
      {Global well-posedness for the 1D cubic nonlinear Shcr\"odinger equation in $L^p,\,p>2$}
\author[Ryosuke Hyakuna]{Ryosuke Hyakuna}
\address[Ryosuke Hyaukna]{Polytechnic University of Japan}
\email{107r107r@gmail.com}
%\address[a]{a}
%\email{a@a}
\begin{abstract}
In this paper, we show that the one dimensional cubic nonlinear Schr\"odinger equation
is globally well posed in $L^p$ for $2\le p <13/6$.  In particular, we prove that the global
solution enjoys the persistence
property for a twisted variable at any time, which implies the result is a natural exetension of the classical global well-posedness in $L^2$ to $L^p$. The proof exploits the data-decomposition argument originally developed by Vargas-Vega in the functional framework introduced by Zhou. 

\end{abstract}
\maketitle
%%%%%%%%%%%%%%%%%%%%%%%%%%%%%%%%%%%%%%%%%%%%%%%%%%%%%%%%%%%%%%%%%%%%%%%%%%
%
%          TEXT START
%
%%%%%%%%%%%%%%%%%%%%%%%%%%%%%%%%%%%%%%%%%%%%%%%%%%%%%%%%%%%%%%%%%%%%%%%%%%

\section{Introduction}
In this paper we consider the following Cauchy problem for the one dimensional cubic nonlinear Schr\"odinger equation:
\begin{equation}
  \label{NLS}  iu_t+u_{xx}+|u|^2u=0,\quad u|_{t=0}=\phi.
\end{equation}
First of all, it is well known that (\ref{NLS}) is locally and globally well posed in $L^2$.  This is a classical result which goes back to Y. Tsutsumi \cite{yT}.  Then it is natural to ask whether the local and global well-posedness in $L^2$ can extend to $L^p,\,p\neq 2$.  This is not an easy question.  One difficulty is the fact that when $p\neq 2$, the $L^p$ regularity is not propagated even in the linear case when $p\neq 2$; that is, one cannot have $U(t)\phi \in L^p$ for $\phi \in L^p,p\neq 2$ in general, where $U(t)$ is the Schr\"odinger evolution group $U(t)\triangleq e^{it\pa_x^2}$.  
This implies that one cannot expect the persistence property of solutions for (\ref{NLS}) when $p\neq 2$.  Moreover, even if some local results are established for data in $L^p,\,p\neq 2$, it is not easy to extend the local solution globally, since one cannot rely directly on conservation laws, which are usually keys to the global existence, as they are usually obtained 
 for data characterized by some kind of square integrability.  Although there have not been many earlier studies in this direction, perhaps, mainly due to these reasons, remarkable progress was made by Zhou.  In \cite{Zhou} he considered the case $1<p<2$ 
and proved that for any data $\phi \in L^p$ there exists a local solution of (\ref{NLS}) such that $v(t)\triangleq U(-t)u(t) \in L^p$. His result implies that the good quantity to follow in $L^p$ is $v(t)$, not the original solution $u$.  Note also that when $p=2$, $U(-t)u(t) \in L^2$ is equivalent to $u(t) \in L^2$.  So in this sense the well-posedness in $L^2$ with the persistence property of solution can be extended to $L^2$ in a natural manner.  Based on this idea the usual notion of well-posedness for (\ref{NLS}) in a data space $E$ can be reformulated as follows.  Let $I\subset \R$ be an interval and let $E$ be a
 Banach space of complex-valued functions on $\R$.  We define
 \begin{equation*}
C_{\mathfrak{S}} (I \,; E) \triangleq 
\{ U(t)v(t)\,|\, v \in C(I\,; E)\,\}.
 \end{equation*}

\begin{defn} \label{LWPdef}
Let $E$ be a Banach space of complex valued functions on $\R$.  (\ref{NLS}) is locally well posed in $E$ if for any $M>0$, there exists $T(M)>0$ such that, for any $\phi \in E$ with $\|\phi \|_E\le M$, there exists a unique solution $u:[0,T(M)] \times \R \to \C$ of (\ref{NLS}) satisfying
\begin{equation*}
    u \in C_{\mathfrak{S}} ([-T(M),T(M)]\,; E)\cap Z_{T(M)} \triangleq S_{T(M)},
\end{equation*}
where $Z_T$ is a space of complex valued functions on $[-T,T]\times \R$. Moreover, 
the map $\phi \mapsto u$ is continuous from
$\{\phi \in E\,|\, \|\phi \|_E\le M\,\}$ to $S_{T_M}$.  Furthermore, (\ref{NLS}) is globally well posed in $E$ if it is locally well posed in $E$ and the local solution $u$ extends to a global one such that $u \in C_{\mathfrak{S}}(\R\,; E)$ for any $\phi \in E$.
\end{defn}

Zhou's result shows that (\ref{NLS}) is locally well posed in $L^p$ for $1<p<2$ in this sense.  Note also that when $E=L^2$ or $E=H^s$, the well-posedness in $E$ is equivalent to the usual one.  In \cite{107jfa} the author extended Zhou's local results to equations with pure power nonlinearities $N(u) =|u|^{\A-1}u$.  Moreover, it is also shown in \cite{107jfa} that the local solution extends globally if
 $p<2$ is sufficiently close to $2$.

Zhou's work and subsequent results indicate that the local and global well-posed results in $L^2$ naturally extend to $L^p,p<2$.  Then it is also natural to ask whether the well-posedness holds for $p>2$.  This is an interesting question, because $L^p$-spaces with $p>2$ are characterized as a space of functions that decay more slowly at infinity than square integrable functions.  But there are even fewer known results for $p>2$ than $p<2$.  As far as the author knows, the first result in this direction is \cite{DSS}
where the authors prove some global results for data in Bessel potential spaces $H^s_p$ 
for some $s>0$ and $p>2$. See also \cite{scjfa} for similar results for slowly decaying data. For mere $L^p(=H^0_p),p>2$ as data space, in \cite{107slow} the author recently obtained the local well-posedness in $L^p$ for $2\le p<4$ in the sense of Definition \ref{LWPdef}.  There we used the same function space as originally introduced and used by Zhou \cite{Zhou}, which we denote by $Y_{q,\theta}^p$ in this paper.  The aim of this paper is to extend the local solution obtained in \cite{107slow} globally in the same functional framework.  In particular, we want the persistence property $v \in C(\R ;L^p)$ of the twisted variable $v(t)=U(-t)u(t)$ for the global solution which is equivalent to the usual persistence property of the solution when $p=2$.

The main result of this paper is as follows.  In the statement below, $q'$ denotes the conjugate of $q$, $Q_{\rho}(r)$ is defined by $2/Q_{\rho}(r) +1/r=1-1/\rho$.  $Y^{\rho}_{q,0}(I)$ is a function space on
$I \times \R,\,I\subset \R$ mentioned above and its definition is given in the next section.

\begin{thm}\label{Maintheorem}
Let $2 \le p < 13/6$.  Then (\ref{NLS}) is globally well posed
in $L^p$ in the sense of Definition \ref{LWPdef}.

More precisely, for any $\phi \in L^p$ there exists a unique global solution
$u\in C_{\mathfrak{S}}(\R\,; L^p(\R))$
of (\ref{NLS}) such that
\begin{equation*}
u|_{[-T,T]} \in Y_{2,0}^2 ([-T,T])+ Y^{3-\epsilon}_{(3-\epsilon)',0} ([-T,T])
    \end{equation*}
    for all $T>0$, where $\epsilon >0$ is a sufficiently small constant determined depending only on $\epsilon$.  Moreover,
    \begin{equation}
        u \in C(\R ; L^2 (\R)) \cap L_{loc}^{Q_2(r)}( \R ; L^r(\R)) +
        L^{Q_{3-\epsilon}(r)}_{loc} (\R ; L^r (\R)) \subset L^{Q_{3-\epsilon}(r)}_{loc} (\R ; L^r (\R)). \label{Strichartzregu}
    \end{equation}
    for any $r \in [3,6]$.
\end{thm}

\begin{rem}
\begin{enumerate}
    \item Combining Theorem \ref{Maintheorem} with the global result \cite[Theorem 1.3]{107jfa} for $p<2$, we obtain the global well-posedness of (\ref{NLS}) in $L^p$ for $9/5<p<13/6$.
\item We note that there are several local and global results for non-$L^2$-based spaces other than $L^p$.  In \cite{Grunrock} Gr\"unrock proved that (\ref{NLS}) is locally well posed in the Fourier-Lebesgue spaces $\hat{L}^p$ defined by $\hat{L}^p
\triangleq \{ \phi \in \mathcal{S}' \,| \, \hat{\phi} \in L^{p'} \}$ for $1<p<\infty$.  Note that $L^p\subset \hat{L}^p$ if $p\le 2$ and $\hat{L}^p \subset L^p$ if $p\ge 2$.  
In \cite{NLSModu} the authors established global solutions of (\ref{NLS}) for data in the Modulation spaces $M_{p,p'}$ for $p>2$ sufficiently close to $2$.  Note that $M_{p,p'} \subset L^p$ when $p>2$.

    \end{enumerate}
\end{rem}

This paper is organized as follows.  In section 2, we give the definition of the function space $Y^p_{q,\theta}$ and related spaces.  We note that in \cite{107slow} we proved the local well posedness in $L^p$ for $2<p<4$ by establishing a local solution in this space.  In essence, we show the global well-posedness by extending the local solution given by \cite{107slow} globally 
in the same functional framework.  Next, in this section we prove some non-linear estimates in $Y^p_{q,\theta}$ which we need to construct a global solution. In section 3, we prove Theorem \ref{Maintheorem}.  The most of the section is devoted to the proof of a key global existence which states that a global solution exists for a set $D_{\A,p_0}$ of data that can be decomposed as the sum of an arbitrarily large $L^2$ function and an arbitrarily small $L^{p_0}$ function for some $p_0>2$.  Then the existence part of the global well-posedness in $L^p$ is immediately obtained by showing that $L^p$ is included in $D_{p_0,\A}$ for suitable $p_0,\A$ when $p<13/6$.  This approach is known as the splitting argument originally developed by Vargas--Vega  \cite{VV} and there are several studies (see e.g. \cite{NLSModu,Grunrock, 107jfa, HTT, 107T}) that exploit this method to obtain a global solution of nonlinear Schr\"odinger equations for data in non-$L^2$-based spaces.

\textbf{Notations.} 
\begin{enumerate}
    \item 
$a'$ is the conjugate of $a$, that is $1/a+1/a'=1$.
\item 
For $1\le r <\infty,\,1\le q \le \infty$ and $I \subset [0,\infty)$ we define 
\begin{equation*}
    L^q (I ; L^r)
    \triangleq \{ u :I\times \R \to \C\,
    |\, \|u\|_{L^q(I  ; L^r)} <\infty\,\},
\end{equation*}
where
\begin{equation*}
 \|u\|_{L^q(I ; L^r)} 
 \triangleq \left( \int_I \left( \int_{\R} |u(t,x)|^r dx\right)^{\frac{q}{r}} dt \right)^{\frac{1}{q}},
\end{equation*}
when $q<\infty$ and
\begin{equation*}
 \|u\|_{L^{\infty}(I ; L^r)} 
 \triangleq {{\rm ess}\, \sup}_{t \in I} \ \left( \int_{\R} |u(t,x)|^r dx\right)^{\frac{1}{r}} dt .
\end{equation*}

\item 
The Fourier transform of $f$ is denoted by $\hat{f}$ or $\mathcal{F}f$.  The inverse Fourier transform is denoted by $\mathcal{F}^{-1}$.
\item Let $1\le p \le \infty$ and $2\le r \le \infty$.  The exponent $Q_p(r)$ is defined by
\begin{equation*}
    \frac{2}{Q_p(r)} +\frac{1}{r}=1-\frac{1}{p} \left(= \frac{1}{p'} \right).
\end{equation*}
Note that $(Q_2(r),r)$ is called admissible for which the well-known Strichartz estimate
holds true:
\begin{equation*}
    \|U(t) \phi \|_{L^{Q_2(r)}(\R ; L^r) } \le C\| \phi \|_{L^2}.
\end{equation*}

\item 
$C$ denotes positive constants that may vary from line to line while subscripted ones like $C_j$ are fixed.
\end{enumerate}
Hereafter we only consider the case $t>0$.  The global solution for $t<0$ can be established in a similar manner.

\section{Function spaces and key estimates}\label{function}
\subsection{Function spaces}

We give the definition of the space $Y_{q,\theta}^p$ along with related spaces.  These spaces were first introduced by Zhou in \cite{Zhou}
to show the local well-posedness of (\ref{NLS}) in $L^p,\,1<p<2$.
\begin{defn}
Let $I=[T_1,T_2],$ where $0<T_1<T_2$, and let $1\le p,q < \infty$ and $\theta\in \R$. 
\begin{enumerate}
\item
 The semi-norm $\|\cdot\|_{\tilde{X}^p_{q,\theta}(I)}$ is defined by
\begin{equation*}
\|v\|_{\tilde{X}^p_{q,\theta}(I)} \triangleq \left(\int^{T_2}_{T_1} \left( s^{\theta} \| (\pa_s v) (s,\cdot)   \|_{L^p}\right)^q ds \right)^{\frac{1}{q}},
\end{equation*}
for any $v$ such that $\|v\|_{\tilde{X}^p_{q,\theta}(I)} <\infty$.

The space $X^p_{q,\theta}(I)$ is defined by
\begin{equation*}
X^p_{q,\theta}(I) \triangleq \{v \in C([T_1,T_2] ;L^p(\R))\,| \, \|v\|_{X^p_{q,\theta}(I)} <\infty \}
\end{equation*}
equipped with the norm
\begin{equation*}
\|v\|_{X^p_{q,\theta}(I)} \triangleq \|v(T_1)\|_{L^p}+\|v\|_{\tilde{X}^p_{q,\theta}(I)}.
\end{equation*}

\item
The space $Y^p_{q,\theta}(I)$ is defined by
\begin{equation*}
Y^p_{q,\theta}(I)\triangleq \{U(t)v(t)\,| \, v \in X^p_{q,\theta}(I)  \},
\end{equation*}
equipped with the norm
\begin{equation*}
\|u\|_{Y^p_{q,\theta}(I)} \triangleq \|U(-t)u(t)\|_{X^p_{q,\theta}(I)}.
\end{equation*}
The semi-norm  $\|\cdot\|_{\tilde{Y}^p_{q,\theta}(T)}$ is defined by
\begin{equation*}
\|u\|_{\tilde{Y}^p_{q,\theta}(I)} \triangleq \| U(-t)u(t)\|_{\tilde{X}^p_{q,\theta}(I)}.
\end{equation*}
for any $u \in Y^p_{q,\theta}(I)$.
\end{enumerate}
\end{defn}

It is not difficult to show that the spaces $X^p_{q,\theta}(T),\,Y^p_{q,\theta}(T)$ are Banach spaces if $-\infty <q'\theta <1$.  
In particular arguing as in the proof of \cite[Lemma 2.1]{107indiana} it is easily checked that the limit of every Cauchy sequence in the
$X^p_{q,\theta}(T)$-norm belongs to $C([0,T] ; L^p)$.  Similarly, it can be shown that $X_{q,\theta}^p(T)$ is continuously embedded in $C([0,T] ;L^p)$ and hence $Y^p_{q,\theta}(T)\hookrightarrow C_{\mathfrak{S}}([0,T]; L^p (\R))$.

\textbf{Notations.} In what follows, we only treat these function spaces with $\theta=0$.  So we simply denote $X^p_{q,0}$ and $Y^p_{q,0}$ by
$X^p_q$ and $Y^p_q$.  Also, when $T_1=0$, we sometimes write $X^p_{q,\theta}(T_2)$ and $Y^p_{q,\theta}(T_2)$ instead of
$X^p_{q,\theta}([0,T_2])$ and $Y^p_{q,\theta}([0,T_2])$.

 \subsection{Nonlinear estimates}
We present key nonlinear estimate we 
need in order to establish a local and global solution of (\ref{NLS}) in the next section.  These estimates are obtained as corollaries of the following generalized Strichartz estimate.

 \begin{lem}{\rm(\cite[Theorem 5]{107T})} \label{odGFS}
Let $2 \le p <4$ and let $r \ge 2$ be such that
\begin{equation}
   \label{conditionodGFS} 0<\frac{1}{Q_p (r)} <\min \left( \frac{1}{4},\frac{1}{2}-\frac{1}{r}\right).
\end{equation}

Then the estimate
\begin{equation}
\| U(t) \phi \|_{L^q(\R ; L^r(\R)} \le C\|\hat{\phi}\|_{L^p}. \label{ofdStr}
\end{equation}
holds true.  In particular, the estimate
\begin{equation}
\| t^{-\frac{2}{q} }U(1/4t) \phi \|_{L^q(\R^{+} ; L^r(\R))} \le C\|\hat{\phi}\|_{L^p}.
\end{equation}
holds true.
\end{lem}

\begin{rem}
The case $p=2$ of (\ref{ofdStr}) coincides with the well-known Strichartz estimate.  We note that (\ref{ofdStr}) is used in \cite{Grunrock,CVV,107T} where the authors construct solutions to nonlinear Schr\"odinger equations with data in the Fourier-Lebesgue spaces.  Note also that the diagonal case $Q_p(r)=r$ of the estimate goes back to \cite{Fefferman}.
\end{rem}

The following inclusion relation can be obtained as a corollary of Lemma \ref{odGFS}.  For $1\le \rho,r <\infty$,\,$\A\in\R$, and $I \subset [0,\infty)$ we define 
\begin{equation*}
    L^{\rho} (I, t^{\A} dt ; L^r)
    \triangleq \{ u :I\times \R \to \C\,
    |\, \|u\|_{L^{\rho}(I ,t^{\A}dt ; L^r)} <\infty\,\},
\end{equation*}
where
\begin{equation*}
 \|u\|_{L^{\rho}(I ,t^{\A}dt ; L^r)} 
 \triangleq \left( \int_I \left( \int_{\R} |u(t,x)|^r dx\right)^{\frac{\rho}{r}} t^{\A}dt \right)^{\frac{1}{\rho}}.
\end{equation*}
\begin{cor}{\rm(\cite[(4.10)]{107slow})}\label{Strinclusion}
Let $2 \le p <4$ and let $r \ge 2$ be such that
\begin{equation}
 \label{Strincluctioncondition}   0<\frac{1}{Q_p (r)} <\min \left( \frac{1}{4},\frac{1}{2}-\frac{1}{r}\right).
\end{equation}
Then the following inclusion holds true for any finite interval $I\subset [0,\infty)$ and $q\in [1,\infty]$ and $\theta \in \R$ with $q'\theta >-1$,
\begin{equation*}
    Y^p_{q,\theta}(I) \subset 
    L^{Q_p(r)}( I ,t^{\frac{1}{p}-\frac{1}{2}}dt; L^r).
\end{equation*}

In particular, 
\begin{equation}
    Y_{p,\theta}^p(T) \subset L^{Q_p(r)}([0,T] ; L^r).
\end{equation}
\end{cor}

\medskip

Now we prove key nonlinear estimates.  Let $0<T_1<T_2$ and $t_0 \in [T_1,T_2]$, and let $Z_j,\,j=1,2,3$ be functions on $[T_1,T_2] \times \R$.  We introduce the trilinear form $\mathscr{D}_{t_0}$ by
\begin{equation*}
\mathscr{D}_{t_0}(Z_1,Z_2,Z_3)\triangleq \int^t_{t_0} U(-s) [(U(s)Z_1(s))( \overline{U(s)Z_2(s)} )(U(s)Z_3(s))] ds.
\end{equation*}
When $t_0=0$ we simply denote $\mathscr{D}_0$ by $\mathscr{D}$.

\begin{prop} \label{trilinear1}
Let $2\le p_0 <4$ and let $I\subset [0,\infty)$ be a finite interval and $t_0 \in I$.
\begin{enumerate}
    \item Assume that $I=[0,T],\,T>0$. Then the following estimates hold true:

\begin{eqnarray}
\| \mathscr{D}(V_1,V_2,W_3) \|_{\tilde{X}^{p_0}_{p_0'}(T)} &\le & C| I |^{\frac{1}{2}-\frac{1}{p_0}}\|V_1 \|_{X_{1,0}^2(T)}\|V_2 \|_{X_{1}^2(T)}\|W_3 \|_{X_{1}^{p_0}(T) } \label{NL1}\\
\| \mathscr{D}(V_1,W_2,V_3) \|_{\tilde{X}^{p_0}_{p_0'}(T)} &\le & C | I |^{ \frac{1}{2}-\frac{1}{p_0} } \|V_1 \|_{X_{1}^2(T)}\|W_2 \|_{X_{1}^{p_0}(I)}\|V_3 \|_{X_{1}^2(T)} \label{NL2}\\
\| \mathscr{D}(V_1,W_2,W_3) \|_{\tilde{X}^{p_0}_{p_0'}(T)} &\le & C|I|^{\frac{3}{4}-\frac{3}{2p_0}}\|V_1 \|_{X_{1}^2(T)}\|W_2 \|_{X_{1}^{p_0}(T)}\|W_3 \|_{X_{1,0}^{p_0}(T)} \label{NL3} \\
\| \mathscr{D}(W_1, V_2,W_3) \|_{\tilde{X}^{p_0}_{p_0'}(T)} &\le & C|I|^{\frac{3}{4}-\frac{3}{2p_0}}\|W_1 \|_{X_{1}^{p_0}(T)}\|V_2 \|_{X_{1}^2(T)}\|W_3 \|_{X_{1}^{p_0}(T)} \label{NL4} \\
\| \mathscr{D}(W_1, W_2,W_3) \|_{\tilde{X}^{p_0}_{p'_0}(T)} &\le & C|I|^{1-\frac{2}{p_0}}\|W_1 \|_{X_{1}^{p_0}(T)}\|W_2 \|_{X_{1}^{p_0}(T)}\|W_3 \|_{X_{1}^{p_0}(T)},\label{NL5}
\end{eqnarray}
where the constant $C$ does not depend on $V_1,V_2,W_j,\,j=1,2,3$ and $I$.
\item Assume that $I=[T_1,T_2],\,0<T_1<T_2$.  Then the following estimates holds true:
\begin{eqnarray}
\| \mathscr{D}_{t_0}(V_1,V_2,W_3) \|_{\tilde{X}^{p_0}_{p_0'}(I)} &\le & C| I |^{ \frac{1}{2}-\frac{1}{p_0} }\|V_1 \|_{X_{1}^2(I)}\|V_2 \|_{X_{1}^2(I)}\|W_3 \|_{X_{1}^{p_0}(I) } \label{NL12}\\
\| \mathscr{D}_{t_0}(V_1,W_2,V_3) \|_{\tilde{X}^{p_0}_{p_0'}(I)} &\le & C | I |^{ \frac{1}{2}-\frac{1}{p_0}} \|V_1 \|_{X_{1}^2(I)}\|W_2 \|_{X_{1}^{p_0}(I)}\|V_3 \|_{X_{1}^2(I)} \label{NL22}\\
\| \mathscr{D}_{t_0}(V_1,W_2,W_3) \|_{\tilde{X}^{p_0}_{p_0'}(I)} &\le & CT_2^{\frac{1}{2}-\frac{1}{p_0}} |I|^{\frac{1}{4}-\frac{1}{2p_0}}\|V_1 \|_{X_{1}^2(I)}\|W_2 \|_{X_{1}^{p_0}(I)}\|W_3 \|_{X_{1}^{p_0}(I)} \label{NL32} \\
\| \mathscr{D}_{t_0}(W_1, V_2,W_3) \|_{\tilde{X}^{p_0}_{p_0'}(I)} &\le & CT_2^{\frac{1}{2}-\frac{1}{p_0}} |I|^{\frac{1}{4}-\frac{1}{2p_0}}\|W_1 \|_{X_{1}^{p_0}(I)}\|V_2 \|_{X_{1}^2(I)}\|W_3 \|_{X_{1}^{p_0}(I)} \label{NL42} \\
\| \mathscr{D}_{t_0}(W_1, W_2,W_3) \|_{\tilde{X}^{p_0}_{p_0'}(I)} &\le & CT_2^{1-\frac{2}{p_0}}\|W_1 \|_{X_{1}^{p_0}(I)}\|W_2 \|_{X_{1}^{p_0}(I)}\|W_3 \|_{X_{1}^{p_0}(I)},\label{NL52}
\end{eqnarray}
where the constant $C$ does not depend on $V_1,V_2,W_j,\,j=1,2,3$ and $I$.

\end{enumerate}

\end{prop}

\begin{proof}
We first assume $I=[0,T]$ and prove (\ref{NL1})--(\ref{NL5}).  Basically, we follow the proof of \cite[Proposition 3.1]{107slow}.  We first recall the well-known factorization formula for $U(-t)$ (see \cite[Chapter 4]{Caz}):
\begin{equation}
\label{factorization0} U(-t)=M_t^{-1}\mathcal{F}^{-1}D_t^{-1} M_t^{-1}
\end{equation}
where the operators $D_t,M_t$ are given by
\begin{equation*}
(D_tw)(x)\triangleq (4\pi i t)^{-\frac{1}{2}}w(x/4 \pi t ),\quad (M_tw)(x)\triangleq e^{\frac{ix^2}{4t}} w.
\end{equation*}
By (\ref{factorization0}) we get the following identity (\cite[(3.12)]{HN}):
\begin{equation}
U(-t)[U(t)Z_1(t)\cdot \overline{U(t)Z_2(t)} \cdot U(t)Z_3(t) ] =ct^{-1}M_t \mathcal{F}^{-1} [
\mathcal{F}M_t Z_1(t) \cdot \overline{\mathcal{F}M_t Z_2(t)} \cdot \mathcal{F}M_t Z_3(t)] \label{factorization}
\end{equation}
where $c$ is an absolute constant.
By (\ref{factorization}) and the Hausdorff-Young and H\"older inequalities we have
\begin{eqnarray*}
\|\pa_t \mathscr{D}(Z_1,Z_2,Z_3) \|_{L^{p_0}} &=& Ct^{-1} \| \mathcal{F}^{-1} [\mathcal{F} M_t Z_1(t) 
\cdot \overline{\mathcal{F}M_t Z_2(t)} \cdot \mathcal{F}M_t Z_3(t) ] \|_{L^{p_0}} \\
& \le & C t^{-1} \| \mathcal{F} M_t Z_1(t) 
\cdot \overline{\mathcal{F}M_t Z_2(t)} \cdot \mathcal{F}M_t Z_3(t)  \|_{L^{p_0'}} \\
&\le & Ct^{-1}\prod_{j=1}^3 \| \mathcal{F} M_t Z_j(t) \|_{L^{3p_0'}}.
\end{eqnarray*}
Note that since $U(1/4t)f=\mathcal{F}^{-1}M_{-t}\mathcal{F}f =\overline{\mathcal{F}M_t \overline{\mathcal{F}f} }$ we have
\begin{equation*}
\mathcal{F}M_t Z_j(t) =\overline{U(1/4t)\mathcal{F}^{-1} \overline{Z_j(t)}}.
\end{equation*}
Let $q\ge 1$.  Taking $L^q([0,T])$-norm of both sides and then using H\"older's inequality in the time variable, we get
\begin{eqnarray*}
\|\pa_t \mathcal{D}(Z_1,Z_2,Z_3) \|_{L^q([0,T] ; L^{3p_0'})} &=&  
\left\|t^{\A} \prod_{j=1}^3 t^{-2/q_j} \| U(1/4t) \mathcal{F}^{-1} \overline{Z_j(t)}  \|_{L^{3p_0'}}\right\|_{L^q([0,T]) }\\
&\le & C\| t^{\A} \|_{L^{q_0}([0,T])} \prod_{j=1}^3 
\| t^{-\frac{2}{q_j}} U(1/4t) \mathcal{F}^{-1} \overline{Z_j(t)} \|_{L^{q_j}([0,T] ; L^{3p_0'})}\\
&\le & CT^{\A+\frac{1}{q_0}} \prod_{j=1}^3 
\| t^{-\frac{2}{q_j}} U(1/4t) \mathcal{F}^{-1} \overline{Z_j(t)} \|_{L^{q_j}([0,T] ; L^{3p_0'})},
\end{eqnarray*}
where $1\le q \le \infty,\,q_0\A>-1$ and
\begin{equation}
\A-2(1/q_1+1/q_2+1/q_3)=-1,\quad 1/q=1/q_0+1/q_1+1/q_2+1/q_3.
\end{equation}
We set $\mathscr{U}(t)\triangleq t^{-\frac{2}{q_j}} U(1/4t) \mathcal{F}^{-1} $ and we write
\begin{equation*}
\overline{Z_j(t)}=\overline{Z_j(0)} +\int^t_0 (\pa_s \overline{Z_j})(s) ds.
\end{equation*}
Then we have
\begin{eqnarray*}
\| \mathscr{U}(t)\overline{Z_j(t)}\|_{L^{q_j} ([0,T] ; L^{3p_0'})}
&\le & C \| \mathscr{U}(t)\overline{Z_j(0)}\|_{L^{q_j} ([0,T] ; L^{3p_0'})}+\left\| \int^t_0
\mathscr{U}(t) (\pa_s \overline{Z_j})(s)  ds \right\|_{L^{q_j} ([0,T] ; L^{3p_0'})} \\
&\le & C \| \mathscr{U}(t)\overline{Z_j(0)}\|_{L^{q_j} ([0,T] ; L^{3p_0'})}+ \left\| \int^t_0 \left\|
\mathscr{U}(t) (\pa_s \overline{Z_j})(s)  \right\|_{L^{3p_0'}}  ds \right\|_{L^{q_j} [0,T]}\\
&\le & C \| \mathscr{U}(t)\overline{Z_j(0)}\|_{L^{q_j} ([0,T] ; L^{3p_0'})}+ \left\| \int^T_0 \left\|
\mathscr{U}(t) (\pa_s \overline{Z_j})(s)  \right\|_{L^{3p_0'}}  ds \right\|_{L^{q_j} [0,T]}\\
&\le & C \| \mathscr{U}(t)\overline{Z_j(0)}\|_{L^{q_j} ([0,T] ; L^{3p'})}+\int^T_0 \left\|
\mathscr{U}(t) (\pa_s \overline{Z_j})(s)  \right\|_{L^{q_j} ([0,T] ; L^{3p_0'})}ds.
\end{eqnarray*}
Let $p_j \in [2,4),\,j=1,2,3$.  If $q_j=Q_{p_j}(3p_0')$ and $(r,p)=(3p_0',p_j)$ satisfies the assumption of Proposition \ref{odGFS}, the right hand side
is estimated by 
\begin{equation}
C\left( \|Z_j (0) \|_{L^{p_j}}+\int^T_0 \| \pa _s Z_j (s) \|_{L^{p_j}} ds \right)\le C \|Z_j \|_{X^{{p_j}}_{1,0}(I)}.
\end{equation}
Thus we get
\begin{equation}
\| \mathscr{D}_{t_0} (Z_1, Z_2, Z_3) \|_{\tilde{X}_{q,0}^{p_0}(I)} \le C T^{\frac{\A+1}{q_0}} \prod_{j=1}^3 
\|Z_j \|_{X_{1,0}^{p_j}(I)},
\end{equation}
for $I=[0,T]$ and for suitable $q_j,p_j,$.

Now we are ready to prove the estimates in question.  Hereafter we set $q=p_0'$.  We first apply the above
argument with
\begin{eqnarray*}
Z_j=V_j\in X_{1}^2(I),\,j=1,2,\quad Z_3=W_3 \in X_{1}^{p_0}(I)
\end{eqnarray*}
and
\begin{equation*}
p_j=2,\quad q_j =Q_{2}(3p_0')=\left( \frac{1}{12}+\frac{1}{6p_0} \right)^{-1},\quad j=1,2,\quad p_3=p_0,\quad q_3=Q_{p_3}(3p_0')=3p_0'.
\end{equation*}
Then $\A=0,\,q_0^{-1}=1/2-1/p_0$ and it is easy to check that $(r,p)\triangleq (3p',p_j)$ satisfies the assumption of Proposition \ref{odGFS} .  Thus we get (\ref{NL1}).  (\ref{NL2}) can be proved in a similar manner.

Next we set
\begin{eqnarray*}
Z_1=V_1\in X_{1}^2(I),\,\quad Z_j=W_j \in X_{1}^{p_0}(I),\quad j=2,3
\end{eqnarray*}
and
\begin{equation*}
p_1=2,\quad q_1 =Q_2(3p_0'),\quad p_j=p_0,\quad q_j=Q_{p_0}(3p_0'),\quad j=2,3.
\end{equation*}
Then $\A=1/2-1/p_0,\,q_0^{-1}=1/4-1/2p_0$ and we have (\ref{NL3}).  The proof of (\ref{NL4}) is similar.
\if0
We set
\begin{equation*}
\frac{1}{q_j}=\frac{1}{2} \left( \frac{1}{6}+\frac{1}{3} \right),\,j=1,2,\quad q_3=3p_0',\,\,
\end{equation*}
Then $1/q_1+1/q_2+1/q_3=1/2$, $q_0=1/2-1/p_0,\, \A=1$, and thus we get (\ref{NL4}).  The proof of (\ref{NL5}) is similar.
\fi
Finally, we set $Z_j=W_j \in X_{1}^{p_0}(I)$ and $q_j=3p_0'$ for $j=1,2,3$, then, $q_0=\infty,\,\A=1-2/p_0$, and we get (\ref{NL5}).

The proof of the inequalities in (ii) is similar except that
we need to modify the estimate of the norm $\|t^{\A}\|_{L^{q_0}(I)}$.  Let $I=[T_1,T_2]$.  Arguing as above, we have

\begin{equation}
\| \mathscr{D}_{t_0} (Z_1, Z_2, Z_3) \|_{\tilde{X}_{q}^{p_0}(I)} \le C \|t^{\A}\|_{L^{q_0}(I)} \prod_{j=1}^3 
\|Z_j \|_{X_{1}^{p_j}(I)}.
\end{equation}

For $q_0<\infty$ we have
\begin{eqnarray*}
    \|t^{\A}\|_{L^{q_0}(I)}&=&
    \left( \int^{T_2}_{T_1} t^{\A q_0} dt \right)^{\frac{1}{q_0}} =\frac{1}{(\A q_0+1)^{\frac{1}{q_0}}} (T_2^{\A q_0+1} -T_1^{\A q_0+1} )^{\frac{1}{q_0}} \\
    &\le & T_2^{\A} (T_2-T_1)^{\frac{1}{q_0}},
\end{eqnarray*}
where we used the elementary estimate
\begin{equation*}
    b^{\gamma}-a^{\gamma}
    =\gamma \int^b_a t^{\gamma-1} dt \le \gamma b^{\gamma-1}(b-a)
\end{equation*}
for $0<a<b,\,\gamma>1.$  We get the same estimate for $q_0=\infty$ with the convention that $1/\infty=0$.  Thus we have
\begin{equation}
\| \mathscr{D} (Z_1, Z_2, Z_3) \|_{\tilde{X}_{q}^{p_0}(I)} \le C T_2^{\A} |I|^{\frac{1}{q_0}} \prod_{j=1}^3 
\|Z_j \|_{X_{1}^{p_j}(I)}.
\end{equation} 
for suitable $\A,q_0,p_j$.  Finally, choosing $Z_j,\,j=1,2,3,\,\A,\,q_j,\,j=0,1,2,3$ as in the
proof of (i), we get the estimates in (ii).
\end{proof}

\medskip

We also need $L^2$ estimates of the trilinear forms $\mathscr{D}_{t_0}$ in order to obtain a global solution.

\begin{prop} \label{nonlinearL2est}
Let $2\le p_0 <3$.  Let $I=[T_1,T_2]$ where $0\le T_1 <T_2 <\infty$ and $t_0 \in I$.  Then the following estimates hold true:
\if0
    \item Assume that $I=[0,T],\,T>0$. Then the following estimates hold true:

\begin{align}
\sup_{t\in I} \| \mathscr{D}(V_1,V_2,W_3) \|_{L^2} &\le  C| I |^{ \frac{3}{4}-\frac{1}{2p_0}}\|V_1 \|_{X_{1}^2(I)}\|V_2 \|_{X_{1}^2(I)}\|W_3 \|_{X_{1}^{p_0}(I) } \label{2NL1}\\
\sup_{t \in I}\| \mathscr{D}(V_1,W_2,V_3) \|_{L^2} &\le  C | I |^{\frac{3}{4}-\frac{1}{2p_0}} \|V_1 \|_{X_{1}^2(I)}\|W_2 \|_{X_{1}^{p_0}(I)}\|V_3 \|_{X_{1}^2(I)} \label{2NL2}\\
\sup_{t \in I} \| \mathscr{D}(V_1,W_2,W_3) \|_{L^2} &\le  C|I|^{1-\frac{1}{p_0}}\|V_1 \|_{X_{1}^2(I)}\|W_2 \|_{X_{1}^{p_0}(I)}\|W_3 \|_{X_{1}^{p_0}(I)} \label{2NL3} \\
\sup_{t \in I} \| \mathscr{D}(W_1, V_2,W_3) \|_{L^2} &\le  C|I|^{1-\frac{1}{p_0}}\|W_1 \|_{X_{1}^{p_0}(I)}\|V_2 \|_{X_{1}^2(I)}\|W_3 \|_{X_{1}^{p_0}(I)} \label{2NL4} \\
\sup_{t\in I}\| \mathscr{D}(W_1, W_2,W_3) \|_{L^2} &\le  C|I|^{\frac{5}{4}-\frac{3}{2p_0}}\|W_1 \|_{X_{1}^{p_0}(I)}\|W_2 \|_{X_{1}^{p_0}(I)}\|W_3 \|_{X_{1}^{p_0}(I)}. \label{2NL5}
\end{align}

\item Assume that $I=[T_1,T_2],\,0<T_1<T_2$.  Then the following estimates hold true:

\fi 
\begin{align}
\sup_{t \in I}\| \mathscr{D}_{t_0}(V_1,V_2,W_3) \|_{L^2} &\le  CT_2^{\frac{1}{2}-\frac{1}{p_0}  } | I |^{ \frac{1}{4}+\frac{1}{2p_0} }\|V_1 \|_{X_{1}^2(I)}\|V_2 \|_{X_{1}^2(I)}\|W_3 \|_{X_{1}^{p_0}(I) } \label{22NL1}\\
\sup_{t \in I} \| \mathscr{D}_{t_0}(V_1,W_2,V_3) \|_{L^2} &\le C T_2^{\frac{1}{2}-\frac{1}{p_0}  }| I |^{ \frac{1}{4}+\frac{1}{2p_0} } \|V_1 \|_{X_{1}^2(I)}\|W_2 \|_{X_{1}^{p_0}(I)}\|V_3 \|_{X_{1}^2(I)} \label{22NL2}\\
\sup_{t\in I}\| \mathscr{D}_{t_0}(V_1,W_2,W_3) \|_{L^2} &\le CT_2^{1-\frac{2}{p_0}  }| I |^{ \frac{1}{p_0} }\|V_1 \|_{X_{1}^2(I)}\|W_2 \|_{X_{1}^{p_0}(I)}\|W_3 \|_{X_{1}^{p_0}(I)} \label{22NL3} \\
\sup_{t\in I}\| \mathscr{D}_{t_0}(W_1, V_2,W_3) \|_{L^2} &\le  CT_2^{1-\frac{2}{p_0}  }| I |^{ \frac{1}{p_0} }\|W_1 \|_{X_{1}^p(I)}\|V_2 \|_{X_{1}^2(I)}\|W_3 \|_{X_{1}^{p_0}(I)} \label{22NL4} \\
\sup_{t\in I} \| \mathscr{D}_{t_0}(W_1, W_2,W_3) \|_{L^2} &\le  CT_2^{\frac{3}{2}-\frac{3}{p_0}  }| I |^{ -\frac{1}{4}+\frac{3}{2p_0} }\|W_1 \|_{X_{1}^{p_0}(I)}\|W_2 \|_{X_{1}^{p_0}(I)}\|W_3 \|_{X_{1}^{p_0}(I)}. \label{22NL5}
\end{align}

\end{prop}

\begin{proof}
Let $t \in [T_1,T_2]$.  By (\ref{factorization}), Plancherel's theorem, and H\"older's inequality, We have
\begin{align*}
    \|\mathscr{D}_{t_0}(Z_1,Z_2,Z_3)(t) \|_{L^2}
    &\le C\int^{T_2}_{T_1} 
    s^{-1} 
    \| \mathcal{F}^{-1} [\mathcal{F} M_t Z_1(t) 
\cdot \overline{\mathcal{F}M_t Z_2(t)} \cdot \mathcal{F}M_t Z_3(t) ] \|_{L^2}  ds\\
&=C \int^{T_2}_{T_1} 
    s^{-1} 
    \|  \mathcal{F} M_t Z_1(s) 
\cdot \overline{\mathcal{F}M_t Z_2(s)} \cdot \mathcal{F}M_t Z_3(s)  \|_{L^2}  ds\\
&\le C \int^{T_2}_{T_1} s^{-1} \prod_{j=1}^3
\|\mathcal{F}M_t Z_j (s)  \|_{L^6} ds \\
&\le C\| s^{\A} \|_{L^{q_0}([T_1,T_2])} \prod_{j=1}^3 
\| t^{-\frac{2}{q_j}} U(1/4s) \mathcal{F}^{-1} \overline{Z_j(s)} \|_{L^{q_j}([T_1,T_2] ; L^6)},
\end{align*}
where
\begin{equation*}
\A-2\left(\frac{1}{q_1}+\frac{1}{q_2}+\frac{1}{q_3} \right)=-1,\quad 1=\frac{1}{q_0}+\frac{1}{q_1}+\frac{1}{q_2}+\frac{1}{q_3}.
\end{equation*}
Now arguing as in the proof of Proposition \ref{trilinear1} we have
\begin{equation}
\|\mathscr{D}_{t_0}(Z_1,Z_2,Z_3)(t) \|_{L^2} \le C T_2^{\A} |I|^{\frac{1}{q_0}} \prod_{j=1}^3 
\|Z_j \|_{X_{1}^{p_j}(I)},
\end{equation} 
as long as $q_j=Q_{p_j}(6)$ and $(p,r)=(p_j,6)$ satisfies the assumption of Proposition \ref{odGFS}.  It is easy to see that the condition is fulfilled as long as $2\le p_j<3$.  We consider the following cases:
\begin{itemize}
\item 
$Z_j=V_j \in X^2_1(I),\,j=1,2$,\quad $Z_3=W_3 \in X_1^{p_0}$  and
\begin{equation*}
    p_j=2,\,\, q_j=Q_2(6)=6,\,\,j=1,2,\quad p_3=p_0,\,\, q_3 =Q_{p_0}(6)=\left( \frac{5}{12}-\frac{1}{2p_0} \right)^{-1}.
\end{equation*}
Then $\A=1/2-1/p_0$ and $1/q_0=1/4+1/2p_0.$  Hence (\ref{22NL1}). The proof of (\ref{22NL2}) is similar. 

\item 
$Z_1=V_1 \in X^2_1(I),$ \quad  $Z_j=W_j \in X_1^{p_0},\,\,j=2,3$, and
\begin{equation*}
    p_1=2,\,\, q_1=Q_2(6),\quad p_j=p_0,\,\, q_j =Q_{p_0}(6),\,\,j=2,3.
\end{equation*}
Then $\A=1-2/p_0$ and $1/q_0=1/p_0.$  Hence (\ref{22NL3}). The proof of (\ref{22NL4}) is similar. 
\item $Z_j=W_j,\,\,j=1,2,3$, and
\begin{equation*}
    p_j=p_0,\,\,q_j=Q_{p_0}(6),\,\,j=1,2,3.
\end{equation*}
Then $\A=3/2-3/p_0$ and $1/q_0=-1/4+3/2p_0$.  We get (\ref{22NL5}).
\end{itemize}

\end{proof}

Finally, we need an estimate for the $L^2$-solution of (\ref{NLS}) in the $Y^2_2$ space.

\begin{prop} \label{L2IVP}
Let $\varphi \in L^2(\R)$ and $t_0 \ge 0$.  Then 
there exist a $T_0>0$ and a solution $v \in Y^2_2([t_0,t_0+T_0])$ of the Cauchy problem

\begin{eqnarray}
\left\{ 
\begin{array}{l}
iv_t +v_{xx} +|v|^2v =0, \quad (t,x)\in [t_0, t_0+T_0] \times \R,\\
v(t_0,x) =\varphi (x),\quad x\in \R. \label{NLSL2}
\end{array}
\right. 
\end{eqnarray}

Moreover, $\|v(t,\cdot) \|_{L^2}=\|\varphi\|_{L^2}$ for all $t \in [t_0,t_0+T_0]$ and there are two absolute constants
$K_1,K_2>0$ such that $T_0=(K_2\|\varphi\|_{L^2})^{-4}$ and
\begin{equation}
\|v\|_{\tilde{Y}_{2}^2([t_0,t_0+T_0])} \le K_1 \|\varphi\|_{L^2}^3 \label{L2est}
\end{equation}
for any $\delta \in [0, (K_2 \|\varphi \|_{L^2})^{-4}]$.

\end{prop}

\begin{proof}
Although the local existence in the $Y^2_2$ space is essentially obtained as a special case of the main theorems in \cite{107slow,Zhou} and most of the other assertions can be deduced from the fixed point argument in the proof therein, we prove it here for completeness.  We first assume $t_0=0$.

We establish a solution $v$ in a closed subset of $Y_{2}^2$.  We set
\begin{equation*}
(\Phi v)(t) \triangleq U(t) \varphi +i\int^t_0 U(t-s) |v|^2 vds
\end{equation*}
and $\mathscr{V}(a,T_0)\triangleq \{v \in Y_{2}^2(T_0)\,|\, v(0)=\varphi,\,\|v\|_{\tilde{Y}^2_{2}(T_0)} \le a\}
$ for $a,T_0 >0$.  We introduce a distance on $\mathscr{V}(a,T_0)$ by 
\begin{equation*}
    d(v_1,v_2) \triangleq \|v_1-v_2\|_{\tilde{Y}^2_2(T_0)}.
\end{equation*}
Then we show that $\Phi:\mathscr{V}(a,T_0)\to \mathscr{V}(a,T_0)$ is well defined and is a contraction mapping for a suitable choice of $a,T_0$.  Hereafter, we let
$V=U(-t)v,\,V_j=U(-t)v_j$ for $v,v_j$.  We first note that
\begin{equation*}
    U(-t)\Phi v=\varphi +c\mathscr{D}(V,V,V).
    \end{equation*}
By (\ref{NL5}) we have
\begin{eqnarray*}
\|\Phi v\|_{\tilde{Y}_{2}^2(T_0)} &=&
\|U(-t)\Phi v\|_{\tilde{X}_{2}^2(T_0)}  \\
&=&c \|\mathscr{D} (V,V,V) \|_{\tilde{X}_{2}^2(T_0)}  \\
&\le &C \| V\|_{\tilde{X}_{1}^2(T_0)}^3 \\
&=&C \| v\|_{\tilde{Y}_{1}^2(T_0)}^3.
\end{eqnarray*}
By H\"older's inequality in the time variable, we have
\begin{eqnarray*}
    \| v\|_{\tilde{Y}_{1}^2(T_0)}^3 &\le & C
    (\|\varphi \|_{L^2}+ T_0^{\frac{1}{2}} \|v\|_{\tilde{Y}^2_{2}(T_0)} )^3 \\
    &\le & 8C \|\varphi \|_{L^2}^3 +8CT_0^{\frac{3}{2}} 
    \|v\|_{\tilde{Y}_{2}^2(T_0)}^3 \\
    &\le & 8C\|\varphi \|_{L^2}^3+8CT_0^{\frac{3}{2}} a^3.
\end{eqnarray*}
Now we put
\begin{equation}
a=16C\|\varphi\|_{L^2}^3,\quad T_0=\varepsilon \|\varphi \|_{L^2}^{-4}, \label{aT0def}
\end{equation}
where $\varepsilon>0$ is a sufficiently small, absolute constant.  Then we have
\begin{eqnarray*}
    \|\Phi v\|_{\tilde{Y}_{2}^2(T_0)} &=& 8C\|\varphi \|_{L^2}^3+8\cdot 16^3 C^4 \varepsilon^{\frac{3}{2}}\|\varphi\|_{L^2}^3 \\
    &\le & \frac{a}{2}+\frac{a}{2}\\
    &=& a
\end{eqnarray*}
if $\varepsilon$ is sufficiently small.  This implies that
$\Phi : \mathscr{V}(a,T_0)\to \mathscr{V}(a,T_0)$ is well defined.  Similarly, we have for $v_1,v_2 \in \mathscr{V}(a,T_0)$
\begin{eqnarray*}
    \|\Phi v_1 -\Phi v_2 \|_{\tilde{Y}_{2}^2(T_0)}
    &=& c \|\mathscr{D}(V_1,V_1,V_1) -\mathscr{D}(V_2,V_2,V_2)
    \|_{\tilde{X}_{2}^2(T_0)} \\
    &\le & c \|\mathscr{D}(V_1-V_2,V_1,V_1) 
    \|_{\tilde{X}_{2}^2(T_0)}
    +c \|\mathscr{D}(V_2,V_1-V_2,V_1) 
    \|_{\tilde{X}_{2}^2(T_0)}\\ 
    &&+ c \|\mathscr{D}(V_1-V_2,V_1,V_1-V_2)
    \|_{\tilde{X}_{2}^2(T_0)}\\
    &\le & CT_0^{\frac{1}{2}} \|V_1-V_2\|_{\tilde{X}_{2}^2(T_0)} \\
    && \times \sum_{1\le j,k \le 2}
    (\|\varphi\|_{L^2}+T_0^{\frac{1}{2}} \|V_j\|_{\tilde{X}_{2}^2(T_0)}) 
    (\|\varphi\|_{L^2}+T_0^{\frac{1}{2}} \|V_k\|_{\tilde{X}_{2}^2(T_0)}) \\
    &\le &  8CT_0^{\frac{1}{2}}(\|\varphi\|_{L^2}^2
+ T_0 a^2)
    \|v_1-v_2\|_{\tilde{Y}_{2}^2(T_0)} \\
    &=& 8C\varepsilon^{\frac{1}{2}} (1+256C^2\varepsilon) \|v_1-v_2\|_{\tilde{Y}_{2}^2(T_0)}.
\end{eqnarray*}
Since $\varepsilon$ is small enough, we have
\begin{equation*}
    8C\varepsilon^{\frac{1}{2}} (1+256C^2\varepsilon) <\frac{1}{2},
\end{equation*}
which implies that $\Phi$ is a contraction mapping.  Consequently, we can show the existence of a local solution $v\in Y_{2}^2(T_0)$ of (\ref{NLSL2}) by the fixed point argument.  In particular, we have $v \in \mathscr{V}(a,T_0)$ and (\ref{L2est}) follows from (\ref{aT0def}).  The case where $t_0>0$ can be proved in almost the same manner; we use (\ref{NL52}) instead of (\ref{NL5}) to estimate the corresponding integral equation.  Finally, by Corollary \ref{Strinclusion} we see that
$v \in C([t_0,t_0+T] ; L^2(\R)) \cap L^8([t_0,t_0+T_0] ; L^4(\R))$ from which we see that by the uniqueness in the Strichartz space, the solution $v$ coincides with the one given by the classical result \cite{yT}.  Hence the $L^2$-conservation holds.

\end{proof}

\section{Proof of Theorem \ref{Maintheorem}}
In this section we prove the main result.  The most of the section is devoted to the proof of a global existence theorem for a data set $D_{p_0,\A}$ whose definition will be given below.  The existence part of Theorem \ref{Maintheorem} can be obtained by showing 
that $L^p$ is included in $D_{p_0,\A}$ for suitable choice of $p_0,\A$ if $p>2$ is sufficiently small.
\subsection{A key global existence theorem}

We first give the definition of the set $D_{p_0,\A}$. 
\begin{defn}
Let $\A>0$ and $1<p<\infty$.  The set $D_{p,\A}$ is 
defined as follows.\, $\phi \in D_{p,\A}$ if and only if there are sequences
$(\varphi_N)_{N>1} \subset L^2$ and $(\psi_N)_{N>1} \subset L^p$ such that
\begin{equation}
    \|\varphi_N \|_{L^2} \le C_0 N^{\A} \label{datadecomp1}
\end{equation}
and
\begin{equation}
    \|\psi_N \|_{L^p} \le \frac{C_0}{N},\label{datadecomp2}
\end{equation}
where $C_0>0$ is independent of $N$.
\end{defn}

We have the following global existence result for data in $D_{p_0,\A}$.

\begin{thm}\label{keyGWP}
    Let $2\le p_0 <3$.  Assume that
    \begin{equation}
        \A <\frac{p_0}{2(2p_0-1)}. \label{globalcond}
    \end{equation}
    Then, for any $\phi \in D_{p_0,\A}$ there exists a unique 
    global solution $u$ of (\ref{NLS}) such that 
    \begin{equation}
        u|_{[0,T]} \in Y_{2}^2(T)+Y_{p_0'}^{p_0}(T) \label{Yregularity}
    \end{equation}
    for any $T>0$, and
    \begin{equation*}
        u \in C([0,\infty) ; L^2(\R)) \cap L_{loc}^{Q_2(r)}([0,\infty) ; L^r(\R))
        +L_{loc}^{Q_{p_0(r)}} ([0,\infty) ; L^r(\R)) \label{Strregularity}    \end{equation*}
    for any $r\ge 2$ satisfying
    \begin{equation}
        0<\frac{1}{Q_{p_0
    }(r)} < \min \left( \frac{1}{4} ,\frac{1}{2}-\frac{1}{r} \right).
        \label{StrregularityScaling}
    \end{equation}
\end{thm}

\subsection{Proof of Theorem \ref{keyGWP} }
 \textbf{Notation.}  Throughout the proof we write $V(t)=U(-t)v(t),\,W(t)=U(-t)w(t),\,V_j(t)=U(-t)v_j(t),\,W_j(t)=U(-t)w_j(t),\,V^{(j)}=U(-t)v^{(j)}(t),\,W^{(j)}=U(-t) w^{(j)}(t)$ for functions 
 $v,w,v_j,w_j,v^{(j)},w^{(j)}$.  The proof of Theorem \ref{keyGWP} proceeds in 
three steps.

\textbf{Step 1}. In the first step we establish a solution is of (\ref{NLS}) on a small time interval.  Let $\phi \in D_{p_0,\A}$ then there are $(\varphi_N)_{N>1}$ and $(\psi_N)_{N>1}$ satisfying $\phi =\varphi_N +\psi_N$ and (\ref{datadecomp1}) and (\ref{datadecomp2}) for each $N>1$.  In this step we construct a solution assuming
\begin{equation}
    \| \varphi_N \|_{L^2}
    \le 2C_0N^{\A},\label{L2size}
\end{equation}
instead of (\ref{datadecomp1}).  Note that this is slightly weaker than (\ref{datadecomp1}).  We fix $N>1$ and set
\begin{equation}
    \delta_N\triangleq (M(2C_0N^{\A}))^{-4},\label{deltadef}
\end{equation}
where $M>1$ is a sufficiently large constant determined below.  The aim of this step is to construct a solution of (\ref{NLS}) on $[0,\delta_N]$.  We seek the solution $u$ in the form of $u=v+w$ where $v$ solves
\begin{equation}
iv_t+v_{xx}+|v|^2v=0,\quad v|_{t=0}=\varphi_N, \label{2NLS}
\end{equation}
and $w$ is the solution to the difference equation:
\begin{equation}
\label{DNLS} iw_t+w_{xx}+G(v,w)=0,\quad w|_{t=0}=\psi_N,
\end{equation}
where
\begin{align*}
    G(v,w)&\triangleq 
    |v+w|^2(v+w)-|v|^2v \\
    &=v^2\bar{w}+ 2v\bar{v}w+2vw\bar{w}+\bar{v}w^2+w^2\bar{w}.
\end{align*}

Clearly, $u=w+v$ solves (\ref{NLS}).  

We have the following existence result on $[0,\delta_N]$.

\begin{prop} \label{Esmallinterval}
Let $\phi \in D_{p_0,\A}$.  Then there is a local solution $u=v+w\in Y_{2}^2(\delta_N)+Y_{p_0'}^{p_0}(\delta_N)$ such that
\begin{equation*}
    \|w\|_{Y^{p_0}_{p_0'}(\delta_N)} \le N^{-1+\frac{4\A}{p_0}}
    \end{equation*}
  and  
\begin{equation*}
    u(t)=v(t)+U(t)\psi_N +i\int^t_0 U(t-s) G(v,w) ds,
\end{equation*}
    where $v,w$ are solutions of {\rm (\ref{2NLS})}, {\rm (\ref{DNLS})}, respectively.
\end{prop}

\begin{proof}
We first note that by Proposition \ref{L2IVP} a solution $v\in Y^2_2(\delta_N)$ exists and satisfies
\begin{equation}
\|V\|_{\tilde{X}_2^2(\delta_N)} =
\|v\|_{\tilde{Y}_2^2(\delta_N)} \le K_1 \|\varphi_N\|^3 \le K_1(2C_0 N^{\A})^3\label{Vest}
\end{equation}
if $2C_0M>K_2$.
Thus it is enough to establish a solution of (\ref{DNLS}).  We solve the corresponding integral equation
\begin{equation*}
w(t)=U(t)\psi_N+i\int^t_0 U(t-s) G(v,w) ds.
\end{equation*}
We find a fixed point
of the operator defined by
\begin{equation*}
(\Phi w)(t) \triangleq U(t)\psi_N+i\int^t_0 U(t-s) G(v,w) ds. 
\end{equation*}
in a closed subset 
\begin{equation*}
    \mathscr{V} \triangleq \{ w \in Y^{p_0}_{p_0'}(\delta_N)\,
    |\, w(0)=\psi_N,\quad \|w\|_{\tilde{Y}^{p_0}_{p_0'}(\delta_N)} \le 
N^{-1+\frac{4\A}{p_0}} \,\}
\end{equation*}
equipped with the distance
\begin{equation*}
    d(w_1,w_2) \triangleq \|w_1-w_2\|_{\tilde{Y}^{p_0}_{p_0'}(\delta_N)}.
\end{equation*}
We may write
\begin{equation*}
    U(-t)\Phi w(t) =\psi_N +2\mathscr{D}(V,V,W)+\mathscr{D}(V,W,V)
    +2\mathscr{D}(V,W,W)+\mathscr{D}(W,V,W)+\mathscr{D}(W,W,W),
\end{equation*}
where $\mathscr{D}$ is the trilinear form
introduced in Section \ref{function}.
Therefore, we have
\begin{eqnarray*}
    \|\Phi w \|_{\tilde{Y}^{p_0}_{p_0'}(\delta_N) } &=& \|U(-t)\Phi w \|_{\tilde{X}^{p_0}_{p_0'}} \\
    &\le & 2\bigl( \| \mathscr{D}(V,V,W) \|_{ \tilde{X}^{p_0}_{p_0'} }
    +\|\mathscr{D}(V,W,V) \|_{\tilde{X}^{p_0}_{p_0'} (\delta_N) }
    +\|\mathscr{D}(V,W,W)\|_{\tilde{X}^{p_0}_{p_0'} (\delta_N)} \\
    && +\|\mathscr{D}(W,V,W)\|_{\tilde{X}^{p_0}_{p_0'} (\delta_N)} +\|\mathscr{D}(W,W,W)\|_{\tilde{X}^{p_0}_{p_0'} (\delta_N) } \bigr).
\end{eqnarray*}
We estimate the norms in the right hand side.  We first prepare several inequalities that frequently appear below.  By (\ref{L2size}), (\ref{deltadef}), (\ref{Vest}),  and H\"older's inequality in the time variable we easily see that
\begin{align*}
\|V\|_{X_{1}^2(\delta_N)} 
&\le \|\varphi_N \|_{L^2}+ \delta_N^{\frac{1}{2}}
\|V\|_{\tilde{X}^2_{2}(\delta_N)}\\
&\le 2C_0N^{\A} +
    \left\{ \left( M(2C_0 N^{\A} )\right)^{-4}\right\}^{\frac{1}{2}}\times K_1 (2C_0 N^{\A} )^3 \\
    &\le 2C_0N^{\A} (1+M^{-2}K_1)  \\
    &\le 2C_0 N^{\A} (1+K_1).
    \end{align*}
Thus we may write
\begin{equation}
    \|V\|_{X_{1}^2(\delta_N)}  \le A_1 N^{\A},
\label{vest1}
\end{equation}
where $A_1\triangleq 2C_0 (1+K_1).$  Similarly, for $w\in \mathscr{V}$ we have
    \begin{align*}
\|W\|_{X_{1}^{p_0}(\delta_N)} 
&\le \|\psi_N \|_{L^{p_0}}+ \delta_N^{\frac{1}{p_0}}
\|W\|_{\tilde{X}^{p_0}_{p_0'}(\delta_N)} \\
&\le C_0N^{-1}+\{\left(M(2C_0N^{\A}) \right)^{-4} \}^{\frac{1}{p_0}}\times N^{-1+\frac{4\A}{p_0}} \\
&\le (C_0+2^{-\frac{4}{p_0}}C_0^{-\frac{4}{p_0}})N^{-1}
\end{align*}
and we write
\begin{equation}
   \|W\|_{X_{1}^{p_0}(\delta_N)}  \le A_2 N^{-1}, \label{west1}
\end{equation}
where $A_2 \triangleq C_0+2^{-\frac{4}{p_0}}C_0^{-\frac{4}{p_0}}$.  Next we compute powers of $\delta_N$.  Let $\beta \in \N$.  We have

\begin{align}
\delta_N^{(\frac{1}{4}-\frac{1}{2p_0})\beta} &= \{\left(M(2C_0N^{\A}) \right)^{-4} \}^{(\frac{1}{4}-\frac{1}{2p_0})\beta} \\
&= M^{(-1+\frac{2}{p_0})\beta} 
2^{(-1+\frac{2}{p_0})\beta} C_0^{(-1+\frac{2}{p_0})\beta} N^{-\beta \A +\frac{2\beta\A}{p_0}},
\end{align}
for which we write
\begin{equation}
    \delta_N = B_{\beta} M^{(-1+\frac{2}{p_0})\beta} N^{-\beta \A +\frac{2\beta\A}{p_0}},\label{deltaest1}
\end{equation}
where $B_{\beta}=2^{(-1+\frac{2}{p_0})\beta} C_0^{(-1+\frac{2}{p_0})\beta}$.
Now we are ready to estimate the norms.  Here we denote by $C_1$ a constant for which all the inequalities in Proposition \ref{trilinear1} hold.  By (\ref{NL1}), (\ref{vest1}), (\ref{west1}), and (\ref{deltaest1}) we have

\begin{align*}
    \| \mathscr{D}(V,V,W) \|_{ \tilde{X}^{p_0}_{p_0'}(\delta_N) }
    &\le C_1 \delta_N^{\frac{1}{2}-\frac{1}{p_0}} 
    \| V \|_{X_1^{2}(\delta_N)}^2 \| W \|_{X_{1}^{p_0}(\delta_N)}\\
    &\le C_1 B_2 M^{-2+\frac{4}{p_0}}   
    N^{-2\A+\frac{4\A}{p_0}} \times A_1^2N^{2\A} \times A_2 N^{ -1}\\
    &\le C_1A_1^2A_2B_2 M^{-2+\frac{4}{p_0}}   
    N^{-1+\frac{4\A}{p_0}}.
\end{align*}

Using (\ref{NL2}) instead of (\ref{NL1}) the norm $\|\mathscr{D}(V,W,V) \|_{\tilde{X}^{p_0}_{p_0'} (\delta_N) }$ can be estimated in exactly the same manner.  
Similarly, using (\ref{NL3}) or (\ref{NL4}) we see that the norms
$\|\mathscr{D}(W,W,V) \|_{\tilde{X}^{p_0}_{p_0'} (\delta_N) }$ and $\|\mathscr{D}(W,V,W) \|_{\tilde{X}^{p_0}_{p_0'} (\delta_N) }$ are smaller than 
\begin{align*}
 C_1\delta_N^{\frac{3}{4}-\frac{3}{2p_0}} 
    \| V \|_{X_1^{2}(\delta_N)} \| W \|_{X_{1}^{p_0}(\delta_N)}^2 
    &\le C_1 B_3 M^{-3+\frac{6}{p_0}}N^{-3\A+\frac{6\A}{p_0}}\times A_1 N^{\A} \times A_2^2 N^{-2}\\
    &\le C_1A_1 A_2^2B_3M^{-3+\frac{6}{p_0}}
    N^{-2\A+\frac{6\A}{p_0}-2}.
\end{align*}
Finally, for the last norm we have
\begin{align*}
      \| \mathscr{D}(W,W,W) \|_{ \tilde{X}^{p_0}_{p_0'} }
      &\le C_1 \delta_N^{1-\frac{2}{p_0}} \| W\|_{X_1^p(T)}^3 \\
    &\le C_1B_4M^{-4+\frac{8}{p_0}} 
    N^{-4\A+\frac{8\A}{p_0}}\times A_2^3N^{-3} \\
    &\le C_1  A_2^3 B_4
    M^{-4+\frac{8}{p_0}}N^{-4\A+\frac{8\A}{p_0}-3},
\end{align*}
where we used (\ref{NL5}).  Collecting these estimates, noting that $N^{-3-4\A+\frac{8\A}{p_0}},N^{-2-2\A+\frac{6\A}{p_0}} \le N^{-1+\frac{4\A}{p_0}}$ and $M^{-4+\frac{8}{p_0}},M^{-3+\frac{6}{p_0}} \le M^{-2+\frac{4}{p_0}}$ we get
\begin{align*}
   \|\Phi w \|_{\tilde{Y}^{p_0}_{p_0',0 } (\delta_N)} &\le  2C_1 \bigl( 2A_1^2A_2B_2+2A_1A_2^2B_3+A_2^3B_4 \biggr) 
   M^{-2+\frac{4}{p_0}} N^{-1+\frac{4\A}{p_0}}.
\end{align*}
Thus $\Phi :\mathscr{V}\to \mathscr{V}$ is well defined if we take $M>1$ sufficiently large so that
\begin{equation}
      2C_1 \bigl( 2A_1^2A_2B_2+2A_1A_2^2B_3+A_2^3B_4 \biggr) 
   M^{-2+\frac{4}{p_0}}<1. \label{mcond1}
\end{equation}

Similarly, we can prove that
$\Phi$ is a contraction mapping.
Let $w_1,w_2 \in \mathscr{V}$.  We estimate $\Phi w_1 -\Phi w_2$.  Since 
\begin{align*}
    G(v,w_1)-G(v,w_2)
    &= v^2 (\overline{w}_1-\overline{w}_2)+2v\overline{v}(w_1-w_2) +2v\overline{w}_1(w_1-w_2)+2vw_2
   ( \overline{w}_1-\overline{w}_2)\\
   & +\overline{v}w_1(w_1-w_2)+\overline{v}w_2(w_1-w_2)+w_1^2(\overline{w}_1-\overline{w}_2)\\
   &+
    w_1\overline{w}_2 (w_1-w_2)+
    w_2 \overline{w}_2(w_1-w_2),
\end{align*}

$U(-t)[\Phi w_1 -\Phi w_2]$ consists of the following terms: 
\begin{align*}
   & \mathscr{D}(V,W_1-W_2,V),\,\mathscr{D}(V,V,W_1-W_2),\,\mathscr{D}(V,W_i, W_1-W_2),\,\mathscr{D}(V,W_1-W_2,W_i),\\
&\mathscr{D}(W_1-W_2,V, W_i), 
\,\,\mathscr{D}(W_i,W_j,W_1-W_2),\,\mathscr{D}(W_i,W_1-W_2,W_j),
\end{align*}
where $1\le i,j \le 2$.

Before we estimate these terms, we note that
\begin{equation*}
\|W_1-W_2\|_{X_{1}^{p_0}(\delta_N)} 
\le  \delta_N^{\frac{1}{p_0}}
\|W_1-W_2\|_{\tilde{X}^{p_0}_{p_0'}(\delta_N)}
\end{equation*}
and we write
\begin{align*}
    \delta_N^{\frac{1}{2}+\beta (\frac{1}{4}-\frac{1}{2p_0})} &= \{\left(M(2C_0N^{\A}) \right)^{-4} \}^{\frac{1}{2}+\beta(\frac{1}{4}-\frac{1}{2p_0})} \\
&= M^{-2-\beta+\frac{2\beta}{p_0}} 
2^{-2-\beta+\frac{2\beta}{p_0}} C_0^{-2-\beta+\frac{2\beta}{p_0}} 
N^{-2\A-\beta \A +\frac{2\beta\A}{p_0}} \\
&= \tilde{B}_{\beta} 
M^{-2-\beta+\frac{2\beta}{p_0}}
N^{-2\A-\beta \A +\frac{2\beta\A}{p_0}},
\end{align*}
where $\tilde{B}_{\beta}\triangleq 2^{-2-\beta+\frac{2\beta}{p_0}} C_0^{-2-\beta+\frac{2\beta}{p_0}} $.

By (\ref{vest1}), (\ref{west1}), and (\ref{NL1}) we have
\begin{align*}
 \|\mathscr{D}(V,V,W_1-W_2)\|_{X_{p_0}^{p'_0}(\delta_N)} &\le C_1\delta_N^{\frac{1}{2}-\frac{1}{p_0}} 
    \| V \|_{X^2_{1}(\delta_N)}^2 \| W_1-W_2 \|_{X_{1}^{p_0}(\delta_N)}\\
    &\le  C_1\delta_N^{\frac{1}{2}} 
    \| V \|_{X_1^{2}(\delta_N)}^2 \| W_1-W_2 \|_{X_{p_0'}^{p_0}(\delta_N)}\\ 
    &\le C_1 \tilde{B}_{0} M^{-2} N^{-2\A} \times A_1^2 N^{2\A} \times
     \| W_1-W_2 \|_{X_{p_0'}^{p_0}(\delta_N)}\\
    &\le C_1 A_1^2  \tilde{B}_0 M^{-2} 
    \times  \| W_1-W_2 \|_{X_{p_0'}^{p_0}(\delta_N)}.
   \end{align*}
The norm of the term $\mathscr{D}(V,W_1-W_2,V)$ can be estimated in almost the same manner.

Next, by (\ref{NL3}), (\ref{NL4}) the norms $\|\mathscr{D}(V,W_i,W_1-W_2)\|_{ },\,\|\mathscr{D}(V,W_1-W_2,W_i)\|_{},\,
\|\mathscr{D}(W_i,V,W_1-W_2)\|_{}\, \,i=1,2$ are estimated by above by
\begin{align*}
C_1\delta_N^{\frac{3}{4}-\frac{3}{2p_0}} 
    \| V \|_{X_1^{2}(\delta_N)} 
    \| W_i \|_{X_{1}^{p_0}(\delta_N)} &\| W_1-W_2 \|_{X_{1}^{p_0}(\delta_N)}
    \le C_1\delta_N^{\frac{3}{4}-\frac{1}{2p_0}} 
    \| V \|_{X_1^{2}(\delta_N)} 
    \| W_i \|_{X_{1}^{p_0}(\delta_N)} \\
   & \qquad \qquad \qquad \quad   \times  \| W_1-W_2 \|_{X_{p_0'}^{p_0}(\delta_N)}\\  
    \le C_1 \tilde{B}_1 M^{-3+\frac{2}{p_0}} &N^{-3\A+\frac{2\A}{p_0}}  \times A_1 N^{\A}
    \times A_2 N^{-1} \times \| W_1-W_2 \|_{X_{p_0'}^{p_0}(\delta_N)} \\
    \le C_1 A_1 A_2 \tilde{B}_1 M^{-3+\frac{2}{p_0}} &N^{-2\A+\frac{2\A}{p_0}-1}  \| W_1-W_2 \|_{X_{p_0'}^{p_0}(\delta_N)} \\
    \le C_1 A_1 A_2  \tilde{B}_1  M^{-2} &  \| W_1-W_2 \|_{X_{p_0'}^{p_0}(\delta_N)}.
\end{align*}

Finally, $\|\mathscr{D}(W_i,W_j,W_1-W_2)\|_{},\,
\|\mathscr{D}(W_i,W_1-W_2,W_j)\|_{},\,1\le i,j \le 2$ are smaller than
\begin{align*}
      C_1\delta_N^{1-\frac{2}{p_0}} 
    \| W_i \|_{X_{1}^{p_0}(\delta_N)}
    \| W_j \|_{X_{1}^{p_0}(\delta_N)} &
    \| W_1-W_2 \|_{X_{1}^{p_0}(\delta_N)}
    \le  C_1\delta_N^{1-\frac{1}{p_0}} 
    \| W_i \|_{X_{1}^{p_0}(\delta_N)} 
    \| W_j \|_{X_{1}^{p_0}(\delta_N)} \\
    & \qquad \qquad \qquad \qquad \qquad   \times 
    \| W_1-W_2 \|_{X_{p_0'}^{p_0}(\delta_N)}\\
    & \le  C_1 \tilde{B}_2 M^{-4+\frac{4}{p_0}} N^{-4\A+\frac{4\A}{p_0}}\times A_2^2 N^{-2} 
    \| W_1-W_2 \|_{X_{p_0'}^{p_0}(\delta_N)}\\
    & \le C_1 A_2^2 \tilde{B}_2 M^{-4+\frac{4}{p_0}} N^{-4\A+\frac{4\A}{p_0}} 
    \| W_1-W_2 \|_{X_{p_0'}^{p_0}(\delta_N)}\\
    &\le C_1 A_2^2 B_2 M^{-2} \| W_1-W_2 \|_{X_{p_0'}^{p_0}(\delta_N)},
\end{align*}
   where we used (\ref{NL5}).
Consequently, we get
\begin{align*}
\|U(-t)\left[ \Phi w_1 -\Phi w_2 \right] \|_{\tilde{X}_{p_0'}^{p_0}(\delta_N)  }
\le (3C_1A_1^2\tilde{B}_0 +6C_1A_1A_2 \tilde{B}_1
+3C_1 A_2^2 B_2 ) M^{-2}  \| W_1-W_2 \|_{\tilde{X}_{p_0'}^{p_0}(\delta_N)},
\end{align*}
which implies
\begin{equation*}
    \| \Phi w_1 -\Phi w_2  \|_{\tilde{X}_{p_0'}^{p_0}(\delta_N)  } 
    \le \frac{1}{2} \|w_1 -w_2\|_{\tilde{X}_{p_0'}^{p_0}(\delta_N)},
\end{equation*}
if we assume
\begin{equation}
    (3C_1 A_1^2\tilde{B}_0 +6C_1A_1A_2 \tilde{B}_1
+3C_1 A_2^2 B_2 ) M^{-2}<\frac{1}{2}.\label{mcond2}
\end{equation}

Consequently, by the fixed point theorem, there is a solution $w\in \mathscr{V}$ of the (\ref{DNLS}).  In particular, the solution satisfies
\begin{equation*}
\|w\|_{\tilde{Y}^{p_0}_{p_0'}(\delta_N)}
\le N^{-1+\frac{4\A}{p_0}}.
\end{equation*}

\end{proof}

\textbf{Step 2}. 
In the next step we extend the local solution on $[0,\delta_N]$ to a
larger time interval $[0,T_N]$, where
\begin{equation*}
    T_N\sim N^{ \frac{2p_0}{3p_0-2}(2\A+1-\frac{2\A}{p_0}) }.
\end{equation*}
Hereafter we write $I_k=[(k-1)\delta_N,k\delta_N]$.  Let $u=v+w$ be the solution established in Step 1.  We consider the Cauchy problem for an unknown function $u^{(2)}$
\begin{equation}
\label{secondIVP}
\left\{
\begin{aligned}
      i(u^{(2)})_t+(u^{(2)})_{xx}+|u^{(2)}|^2u^{(2)}&=0,\quad t>\delta_N \\
    u^{(2)}|_{t=\delta_N}&=u(\delta_N)=\varphi^{(2)}_N(x)+\psi^{(2)}_N(x),
\end{aligned}
\right.
\end{equation}
where
\begin{equation*}
    \varphi^{(2)}(x)\triangleq v(\delta_N)+i\int^{\delta_N}_0 
    U(\delta_N-s) G(v,w) ds,\quad \psi^{(2)}_N(x)\triangleq U(\delta_N)\psi_N.
\end{equation*}
As in the previous step we split (\ref{secondIVP}) into the two Cauchy problems
\begin{equation}
    i(v^{(2)})_t+(v^{(2)})_{xx}+|v^{(2)}|^2v^{(2)}=0,\quad v^{(2)}|_{t=\delta_N} =\varphi^{(2)}_N
    \label{secondIVPd1}
\end{equation}
and
\begin{equation}
    iw_t^{(2)}+(w^{(2)})_{xx}+G(v^{(2)},w^{(2)})=0,\quad w^{(2)}|_{t=\delta_N} =\psi^{(2)}_N. \label{secondIVPd2}
\end{equation}
   Note that $v^{(2)}+w^{(2)}$ is a solution of (\ref{secondIVP}). Then, as we shall prove below, the argument in the previous step can apply to obtain solutions to (\ref{secondIVPd1}), (\ref{secondIVPd2}) on $I_2$.  
  To be more precise, we will see that
  \begin{equation*}
     \|\varphi^{(2)}_N \|_{L^2} \le (2C_0) N^{\A}
  \end{equation*}
  and thus there is an $L^2$-solution $v^{(2)}$ of (\ref{secondIVPd1}) on $I_2$ such that $\|v^{(2)}(t)\|_{L^2}
  =\|\varphi^{(2)}\|_{L^2},\,\forall t\in I_2$ and
  \begin{equation*}
      \|v^{(2)} \|_{\tilde{Y}_{2}^2(I_2)} =
      \|V^{(2)} \|_{\tilde{X}_{2}^2(I_2)} \le K_1 \|\varphi^{(2)}_N \|_{L^2}^3
      \le K_1(2C_0N^{\A})^3.
  \end{equation*}
  This enables us to establish a solution of (\ref{secondIVPd2}) on $I_2$ can be established in the closed subset
  \begin{equation*}
      \mathscr{V}^{(2)} \triangleq \{ w^{(2)} \in Y^{p_0}_{p_0'} (I_2)\,
      |\,w^{(2)}(\delta_N) =\psi_N^{(2)},\, \|w^{(2)} \|_{ \tilde{Y}^{p_0}_{p_0'} (I_2)} \le N^{-1+\frac{4\A}{p_0}}\,\}.
  \end{equation*}
  by the fixed point argument.  Consequently, we obtain a solution of (\ref{NLS}) on
  the extended interval $I_1\cup I_2=[0,2\delta_N]$.  Similarly, We repeat this argument to extend the local solution to the times $3\delta_N,4\delta_N,\cdots$ as long as the same fixed point argument can be applied. 
 The next proposition insists that we can repeat this procedure $n_0$ times, where
  \begin{equation}
      n_0\triangleq \text{integer part of} \,\,cN^{\frac{2p_0}{3p_0-2}(2\A+1-\frac{2\A}{p_0})} \label{repeattimes}
  \end{equation}
  where $c>0$ is an absolute constant.

\begin{prop}
Let $N>1$ be sufficiently large and let $n_0$ be given by {\rm(\ref{repeattimes})}.   Then there are inductively defined solutions $u^{(j)}\triangleq v^{(j)}+w^{(j)}
\in Y^2_{2}(I_j)+Y^{p_0}_{p_0'}(I_j),\,\,j=1,2,\cdots,n_0$ of the initial value problems
\begin{equation}
\label{kthIVP}
\left\{
\begin{aligned}
      iu^{(j)}_t+u^{(j)}_{xx}+|u^{(j)}|^2u^{(j)}&=0,\quad t>(j-1)\delta_N \\
    u^{(j)}|_{t=(j-1)\delta_N}&=u^{(j-1)}((j-1)\delta_N)=\varphi^{(j)}_N(x)+\psi^{(j)}_N(x),
\end{aligned}
\right.
\end{equation}
where 
\begin{equation*}
    \varphi^{(j)}(x)\triangleq v^{(j-1)}(\delta_N)+i\int^{(j-1)\delta_N}_{(j-2)\delta_N}
    U((j-1)\delta_N-s) G(v^{(j-1)},w^{(j-1)}) ds,
\end{equation*}
for $j\ge 2$ and $\varphi^{(1)}\triangleq \varphi_N$, and $\psi^{(j)}\triangleq U((j-1)\delta_N) \psi_N$.

Moreover, $u^{(j)}$ can be decomposed as $u^{(j)}= v^{(j)}+w^{(j)}
\in Y^2_{2}(I_j)+Y^{p_0}_{p_0'}(I_j)$, where
$v^{(j)}$ solves 
\begin{equation}
\label{kthIVP1} i(v^{(j)})_t+(v^{(j)})_{xx}+|v^{(j)}|^2v^{(j)}=0,\quad  v^{(j)}|_{t=(j-1)\delta_N}=\varphi^{(j)}_N,
\end{equation}
and 
\begin{equation}
\|v^{(j)}\|_{\tilde{Y}_{2}^2(I_j) } \le K_1(2C_0 N^{\A})^3, \label{kthIVP1prop}
\end{equation}
and $\|v^{(j)}(t)\|_{L^2}=\|\varphi^{(j)}\|_{L^2}\,\forall t \in I_j$, and\, $w^{(j)}$ solves 
\begin{equation}
\label{KthIVP} i(w^{(j)})_t+(w^{(j)})_{xx}+G(v^{(j)},w^{(j)})=0,\quad  w^{(j)}|_{t=(j-1)\delta_N}=\psi^{(j)}_N,
\end{equation}
and
\begin{equation}
  \| w^{(j)} \|_{ \tilde{Y}^{p_0}_{p_0'} (I_j) } \le N^{-1+\frac{4\A}{p_0}} . \label{kthIVPprop2}
\end{equation}
\end{prop}

\begin{proof}  The assertion for $j=1$ was proved in Step 1.  Assume that $u^{(j)},\,j=1,\cdots,k-1$ with the desired properties are defined.  We show that $u^{(k)}$ is defined if $k\le n_0$.  We first show that $\|\varphi^{(k)} \|_{L^2} \le (2C_0)N^{\A}$.  By the definition of $\varphi^{(j)}_N,\,j=1,\cdots,k$, the $L^2$-conservation law, and (\ref{datadecomp1}), we have
\begin{align*}
    \|\varphi^{(k)}_N\|_{L^2} &\le  \| v^{(k-1)}(k\delta_N)\|_{L^2} +\left\| \int^{(k-1)\delta_N}_{(k-2)\delta_N} 
    U((k-1)\delta_N-s) G(v^{(k-1)},w^{(k-1)}) ds\right\|_{L^2} \\
    &= \|\varphi^{(k-1)}_N \|_{L^2}+\left\| \int^{(k-1)\delta_N}_{(k-2)\delta_N} 
    U((k-1)\delta_N-s) G(v^{(k-1)},w^{(k-1)}) ds\right\|_{L^2} \\
    &\le \| v^{(k-2)}((k-1)\delta_N)\|_{L^2} +\left\| \int^{(k-2)\delta_N}_{(k-3)\delta_N} 
    U((k-2)\delta_N-s) G(v^{(k-2)},w^{(k-2)}) ds\right\|_{L^2}  \\
    \qquad +&\left\| \int^{(k-1)\delta_N}_{(k-2)\delta_N} 
    U((k-1)\delta_N-s) G(v^{(k-1)},w^{(k-1)}) ds\right\|_{L^2} \\
    &=\| \varphi^{(k-2)}_N\|_{L^2} +\left\| \int^{(k-2)\delta_N}_{(k-3)\delta_N} 
    U   ((k-2)\delta_N-s) G(v^{(k-2)},w^{(k-2)}) ds\right\|_{L^2}  \\
    \qquad +&\left\| \int^{(k-1)\delta_N}_{(k-2)\delta_N} 
    U((k-1)\delta_N-s) G(v^{(k-1)},w^{(k-1)}) ds\right\|_{L^2} \\
    &\le \cdots \le \|\varphi_N\|_{L^2} +\sum_{j=1}^{k-1} 
    \left\| \int^{j\delta_N}_{(j-1)\delta_N} 
    U(j\delta_N-s) G(v^{(j)},w^{(j)}) ds\right\|_{L^2} \\
    &\le C_0 N^{\A} +\sum_{j=1}^{k-1} 
    \left\| \int^{j\delta_N}_{(j-1)\delta_N} 
    U(j\delta_N-s) G(v^{(j)},w^{(j)}) ds\right\|_{L^2}.
\end{align*}
We estimate the norms in the right hand side.  Using the unitarity property of
$U(t)$ and arguing as in the previous step we have
\begin{align*}
   &\left\| \int^{j\delta_N}_{(j-1)\delta_N} 
    U(j\delta_N-s) G(v^{(j)},w^{(j)}) ds\right\|_{L^2}
    = \left\| \int^{j\delta_N}_{(j-1)\delta_N} 
    U(-s) G(v^{(j)},w^{(j)}) ds\right\|_{L^2} \\
    & \le 2\sup_{t\in I_{j}}\biggl( \|\mathscr{D}_{(j-1)\delta_N} (V^{(j)}, V^{(j)}, W^{(j)})
    \|_{ L^2 }  +\|\mathscr{D}_{(j-1)\delta_N} (V^{(j)}, W^{(j)}, V^{(j)})
    \|_{ L^2 } \\
    &+\|\mathscr{D}_{(j-1)\delta_N} (V^{(j)}, W^{(j)}, W^{(j)})
    \|_{ L^2 }
    +\|\mathscr{D}_{(j-1)\delta_N} (W^{(j)}, V^{(j)}, W^{(j)})
    \|_{ L^2 }  \\
    &+  \|\mathscr{D}_{(j-1)\delta_N} (W^{(j)}, W^{(j)}, W^{(j)})
    \|_{ L^2 }\biggr)
\end{align*}
  Observe first that, arguing as in the 
  estimate of $\|V\|_{X_1^2(\delta_N)}$ and $\|W\|_{X_1^{p_0}(\delta_N)}$ in Step 1, using (\ref{kthIVP1prop}) and (\ref{kthIVPprop2}), we have
\begin{equation}
     \|V^{(j)} \|_{X_{1}^2(I_{j})}\le 
     A_1  N^{\A},\quad \|W^{(j)} \|_{X_{1}^{p_0} (I_j)}
     \le A_2 N^{-1}.
\end{equation}
We denote by $C_2$ a constant for which all the inequalities in Proposition \ref{nonlinearL2est}
hold.  By (\ref{22NL1}) and these estimates we get
\begin{align*}
\|\mathscr{D}_{(j-1)\delta_N} (V^{(j)}, V^{(j)}, W^{(j)})(t)
    \|_{ L^2 }  &\le 
    C_2(j\delta_N)^{\frac{1}{2}-\frac{1}{p_0}}
    \delta_N^{\frac{1}{4}+\frac{1}{2p_0}}
    \|V^{(j)} \|^2_{X_{1}^2(I_{j})}
    \|W^{(j)} \|_{X^{p_0}_{1}(I_{j})} \\
    &= C_2 j^{\frac{1}{2}-\frac{1}{p_0}} 
    \delta_N^{\frac{3}{4}-\frac{1}{2p_0}}
    \|V^{(j)} \|^2_{X_{1}^2(I_{j})}
    \|W^{(j)} \|_{X^{p_0}_{1}(I_{j})}   \\
    &\le C_2 j^{\frac{1}{2}-\frac{1}{p_0}}
    \tilde{B}_1 M^{-3+\frac{2}{p_0}} N^{-3\A+\frac{2\A}{p_0}} \times A_1^2 N^{2\A}\times A_2 N^{-1}\\
    &\le A_1^2 A_2 \tilde{B}_1 C_2 N^{-\A+\frac{2\A}{p_0}-1},
      \end{align*}
for any $t \in I_{j}$.  

Similarly, we get
\begin{equation*}
\sup_{t\in I_{j-1}}\|\mathscr{D}_{(j-1)\delta_N} (V^{(j)}, W^{(j)}, V^{(j)})(t)
    \|_{ L^2 }  \le A_1^2 A_2 \tilde{B}_1 C_2 N^{-\A+\frac{2\A}{p_0}-1}.
\end{equation*}

Next, we use (\ref{22NL3}) to obtain

\begin{align*}
\|\mathscr{D}_{(j-1)\delta_N} (V^{(j)}, W^{(j)}, W^{(j)})(t)
    \|_{ L^2 }  &\le 
    C_2(j\delta_N)^{1-\frac{2}{p_0}}
    \delta_N^{\frac{1}{p_0}}
    \|V^{(j)} \|_{X_{1}^2(I_{j})}
    \|W^{(j)} \|^2_{X^{p_0}_{1}(I_{j})} \\
    & =C_2j^{1-\frac{2}{p_0}} 
    \delta_N^{1-\frac{1}{p_0}}
     \|V^{(j)} \|_{X_{1}^2(I_{j})}
    \|W^{(j)} \|^2_{X^{p_0}_{1}(I_{j})} \\
    &\le C_2\tilde{B}_2 M^{-4+\frac{4}{p_0}}
    N^{-4\A+\frac{4\A}{p_0}}\times A_1 N^{\A} \times A_2^2 N^{-2}  \\
    &\le A_1 A_2^2 \tilde{B}_2 C_2 N^{-3\A+\frac{4\A}{p_0}-2}
    \end{align*}
for any $t \in I_{j}$.  We also get the 
same upper bound for $\|\mathscr{D}_{(j-1)\delta_N} (W^{(j)}, V^{(j)}, W^{(j)})(t)
    \|_{ L^2 } $ using (\ref{22NL4}).  Finally, by (\ref{22NL5}) we have
 \begin{align*}
\sup_{t\in I_{j}}\|\mathscr{D}_{(j-1)\delta_N} (W^{(j)}, W^{(j)}, W^{(j)})(t)
    \|_{ L^2 }  &\le 
    C_2(j\delta_N)^{\frac{3}{2}-\frac{3}{p_0}}
    \delta_N^{-\frac{1}{4}+\frac{3}{2p_0}}
    \|W^{(j)} \|^3_{X^{p_0}_{1}(I_{j})} \\
    &=  C_2 j^{\frac{3}{2}-\frac{3}{p_0}} 
    \delta_N^{\frac{5}{4}-\frac{3}{2p_0}}  \|W^{(j)} \|^3_{X^{p_0}_{1}(I_{j})} \\
    & \le C_2 j^{\frac{3}{2}-\frac{3}{p_0}}  
    \times \tilde{B}_3 M^{-5+\frac{6}{p_0}}
    N^{-5\A+\frac{6\A}{p_0}} \times A_2^3N^{-3}\\
   &\le A_3^3 \tilde{B}_3 C_2 j^{\frac{3}{2}-\frac{3}{p_0}} N^{-5\A+\frac{6\A}{p_0}-3}.
    \end{align*}   
Collecting these estimates we obtain
\begin{align*}
      \left\| \int^{j\delta_N}_{(j-1)\delta_N} 
    U(j\delta_N-s) G(v^{(j)},w^{(j)}) ds\right\|_{L^2} 
    &\le L_0
    \bigl( j^{\frac{1}{2}-\frac{1}{p_0}}N^{-\A+\frac{2\A}{p_0}-1}+  j^{1-\frac{2}{p_0}}N^{-3\A+\frac{4\A}{p_0}-2}\\
   & +  j^{\frac{3}{2}-\frac{3}{p_0}} N^{-5\A+\frac{6\A}{p_0}-3}\bigr)\\
   & \triangleq J_1+J_2+J_3.
\end{align*}
where 
\begin{equation*}
    L_0=\max (4A_1^2 A_2 \tilde{B}_1 C_2, \,\,4A_1 A_2^2 \tilde{B}_2 C_2, \,\,2A_3^3 \tilde{B}_3 C_2,\,\,1).
\end{equation*}
We estimate the sum of $J_1,J_2,J_3$.  We recall the elementary estimate
\begin{equation}
\sum_{j=1}^{k-1} j^s \le \tilde{C}_s k^{s+1} \label{sume}
\end{equation}
for $s>0$, where $\tilde{C_s}$ is a constant depending only on $s$.  In particular, we have
\begin{align}
 \sum_{j=1}^{k-1} j^{\frac{1}{2}-\frac{1}{p_0}} &\le L_1 k^{\frac{3}{2}-\frac{1}{p_0}}\label{ele1}\\
     \sum_{j=1}^{k-1} j^{1-\frac{2}{p_0}} &\le L_1 k^{2-\frac{2}{p_0}} \label{ele2}\\
      \sum_{j=1}^{k-1} j^{\frac{3}{2}-\frac{3}{p_0}} &\le L_1 k^{\frac{5}{2}-\frac{3}{p_0}},
\end{align}
for some $L_1>\max (C_0,1)$ depending only on $p_0$.
We first estimate the sum of $J_1$.  By (\ref{ele1}) we have
\begin{align}
  \sum_{j=1}^{k-1} J_1 = L_0 \sum_{j=1}^{k-1} 
j^{\frac{1}{2}-\frac{1}{p_0}}N^{-\A+\frac{2\A}{p_0}-1} \le L_0L_1 k^{\frac{3}{2}-\frac{1}{p_0}} N^{-\A+\frac{2\A}{p_0}-1}.
\end{align}
Observe that the inequality
\begin{equation}
    k^{\frac{3}{2}-\frac{1}{p_0}} N^{-\A+\frac{2\A}{p_0}-1} \le\frac{C_0}{3L_0 L_1} N^{\A} \label{J1est0}
\end{equation}
is equivalent to
\begin{equation}
k^{\frac{3}{2}-\frac{1}{p_0}} \le 
 \frac{C_0}{3L_0L_1}  N^{2\A-\frac{2\A}{p_0}+1}, \label{J1est1}
\end{equation}
and 
\begin{equation}
    k\le \left(\frac{C_0}{3L_0L_1} \right)^{\frac{2p_0}{3p_0-2}} 
   N^{\frac{2p_0}{3p_0-2} (2\A-\frac{2\A}{p_0}+1)}. \label{J1est2}
\end{equation}
This implies that 
\begin{equation*}
    \sum_{j=1}^{k-1} J_1 \le \frac{C_0 N^{\A}}{3} 
\end{equation*}
if $k\le n_0 \triangleq cN^{ \frac{2p_0}{3p_0-2} (2\A-\frac{2\A}{p_0}+1) }$ and we take a suitable $c>0$.  Next we estimate $\sum J_2,\,\sum J_3$, but these can be controlled by the sum of the first term.  To see this, write
\begin{equation*}
    \sum_{j=1}^{k-1} J_2\le L_0L_1k^{2-\frac{2}{p_0}}
    N^{-3\A+\frac{4\A}{p_0}-2}
\end{equation*}
using (\ref{ele2}) and observe that the estimate
\begin{equation}
 k^{2-\frac{2}{p_0}}
    N^{-3\A+\frac{4\A}{p_0}-2} \le 
      \frac{C_0}{3L_0L_1} k^{\frac{3}{2}-\frac{1}{p_0}} N^{-\A+\frac{2\A}{p_0}-1} \label{J2est1}
\end{equation}
is equivalent to 
\begin{equation}
k^{\frac{1}{2}-\frac{1}{p_0}} \le \frac{C_0}{3L_0 L_1}  N^{2\A-\frac{2\A}{p_0}+1}. \label{J2est1}
\end{equation}
Then, in view of (\ref{J1est1}) we easily see that (\ref{J2est1}) holds true.  Thus we have 
\begin{equation*}
    \sum_{j=1}^{k-1} J_2 \le L_0 L_1 \times \frac{C_0}{3L_0 L_1}k^{ \frac{3}{2}-\frac{1}{p_0}} N^{-\A+\frac{2\A}{p_0}-1} \le 
    \frac{C_0}{3} \times \frac{C_0}{3L_0 L_1}N^{\A} \le \frac{C_0 N^{\A}}{3},
\end{equation*}
where we used (\ref{J1est0}) in the last inequality(note that we chose $L_0,L_1$ so that $L_1>C_0,\,L_0>1$).
Similarly, writing
\begin{equation*}
    \sum_{j=1}^{k-1} J_3 \le L_0 L_1 
    k^{\frac{5}{2}-\frac{3}{p_0}} N^{-5\A+\frac{6\A}{p_0}-3}
\end{equation*}
and noting that the inequality
\begin{equation}
 k^{\frac{5}{2}-\frac{3}{p_0}} N^{-5\A+\frac{6\A}{p_0}-3} \le 
    \frac{C_0}{3L_0 L_1}  k^{\frac{3}{2}-\frac{1}{p_0}} N^{-\A+\frac{2\A}{p_0}-1} \label{J3est1}
\end{equation}
is equivalent to 
\begin{equation}
k^{1-\frac{2}{p_0}} \le \frac{C_0}{3L_0 L_1} N^{2\left(2\A-\frac{2\A}{p_0}+1\right)}, \label{J3est2}
\end{equation}
which is fullfilled by (\ref{J2est1}), we see that 
\begin{equation*}
    \sum_{j=1}^{k-1} J_3 \le L_0 L_1 \times\frac{C_0}{3L_0 L_1} k^{\frac{3}{2}-\frac{1}{p_0}} N^{-\A+\frac{2\A}{p_0}-1} \le 
    \frac{C_0 N^{\A}}{3}.
\end{equation*}
Consequently, we have
\begin{align*}
    \|\varphi^{(k)}_N\|_{L^2} &\le C_0 N^{\A} +\sum_{j=1}^{k-1} 
    \left\| \int^{j\delta_N}_{(j-1)\delta_N} 
    U(j\delta_N-s) G(v^{(j)},w^{(j)}) ds\right\|_{L^2} \\
    &\le C_0 N^{\A} +\frac{C_0 N^{\A}}{3}+\frac{C_0 N^{\A}}{3}+\frac{C_0 N^{\A}}{3} \\
    & =2C_0 N^{\A},
\end{align*}
if $k \le n_0$.  Thus, by Proposition \ref{L2IVP}, there is
a solution $v^{(k)}$ of (\ref{kthIVP1}) satisfying
(\ref{kthIVP1prop}) and the $L^2$ conservation law.  It remains to prove that $w^{(k)}$ is defined on $I_k$.  We define
\begin{equation*}
\Phi^{(k)}w^{(k)}(t)
\triangleq U(t-(k-1)\delta_N) \psi^{(k)}_N
+i\int^t_{(k-1)\delta_N} U(t-s) G(v^{(k-1)},
w^{(k-1)} ) ds
\end{equation*}
and
\begin{equation*}
    \mathscr{V}^{(k)}\triangleq \{ w^{(k)} \in Y^{p_0}_{p_0'}(I_k)\,
    |\, w^{(k)}((k-1)\delta_N)=\psi_N^{(k)},\quad \|w^{(k)}\|_{\tilde{Y}^{p_0}_{p_0'}(I_k)} \le 
N^{-1+\frac{4\A}{p_0}} \,\}.
\end{equation*}

We show that $\Phi^{(k)}:\mathscr{V}^{(k)}\to \mathscr{V}^{(k)}$ is 
well defined and is a contraction mapping.  Let $w^{(k)} \in \mathscr{V}^{(k)}$.  Since
\begin{align*}
    U(-t)\Phi^{(k)}w^{(k)} 
    (t) &= U(-(k-1)\delta_N) \psi^{(k)}_N+i\int^t_{(k-1)\delta_N}
    U(-s)G(v^{(k-1)},w^{(k-1)} ) ds \\
    &= \psi_N +2\mathscr{D}_{(k-1)\delta_N} (V,V,W)+\mathscr{D}_{(k-1)\delta_N }(V,W,V)
    +2\mathscr{D}_{ (k-1)\delta_N } (V,W,W) \\
    &+\mathscr{D}_{ (k-1)\delta_N }( W,V,W) +\mathscr{D}_{ (k-1)\delta_N}(W,W,W),
\end{align*}
the proof is almost the same as that of Step 1.  The only difference is that
one needs to use (\ref{NL12})--(\ref{NL52}) in stead of (\ref{NL1})--(\ref{NL5}).  We estimate the norms of three types of terms in the right hand side. Firstly, note that the norms
$\| \mathscr{D}(V,V,W)\|_{\tilde{X}^{p_0}_{p'_0}(I_k)}$ and 
$\| \mathscr{D}(V,W,V)\|_{\tilde{X}^{p_0}_{p'_0}(I_k)} $ are estimated in the same way as in the Step 1, since the right hand side of (\ref{NL12}) and (\ref{NL22}) depend on $|I|$ (as well as $\|V\|,\|W\|$) but not on its position on the real line. Thus
\begin{align*}
   \| \mathscr{D}(V,V,W) \|_{ \tilde{X}^{p_0}_{p_0'}(I_k) } ,\,
    \| \mathscr{D}(V,W,V) \|_{ \tilde{X}^{p_0}_{p_0'}(I_k) } 
    &\le C_1A_1^2A_2B_2 M^{-2+\frac{4}{p_0}}   
    N^{-1+\frac{4\A}{p_0}}.
\end{align*}
Next, we use (\ref{NL32}) and (\ref{NL42}) to obtain
\begin{align*}
\|\mathscr{D}_{ (k-1)\delta_N (I_k) } (V,W,W)\|_{\tilde{X}^{p_0}_{p'_0}} &\le 
C_1 (k\delta_N)^{ \frac{1}{2}-\frac{1}{p_0}}
\delta_N^{\frac{1}{4}-\frac{1}{2p_0}} 
\| V^{(k)} \|_{X^2_{1}(I_k)} \|W^{(k)} 
\|_{X_{1}^{p_0}(I_k)}^2 \\
&= C_1 k^{ \frac{1}{2}-\frac{1}{p_0}}
\delta_N^{\frac{3}{4}-\frac{3}{2p_0}} 
\| V^{(k)} \|_{X^2_{1}(I_k)} \|W^{(k)} 
\|_{X_{1}^{p_0}(I_k)}^2 \\
&\le C_1 k^{\frac{1}{2}-\frac{1}{p_0}} \times B_3 M^{-3+\frac{6}{p_0}} N^{-3\A+\frac{6\A}{p_0}} \times N^{\A} \times N^{-2}   \\
&\le A_1 A_2^2 B_3 C_1 k^{ \frac{1}{2}-\frac{1}{p_0}} M^{-3+\frac{6}{p_0}} 
N^{-2\A+\frac{6\A}{p_0}-2} .
\end{align*}
Now recalling (\ref{J2est1}) we see that the the right hand side of the above inequalities is smaller than
\begin{equation*}
    A_1 A_2^2 B_3 C_1 M^{-3+\frac{6}{p_0}}
    N^{-2\A+\frac{6\A}{p_0}-2} \times \frac{C_0}{3L_0 L_1} N^{2\A-\frac{2\A}{p_0}+1} \le\frac{ A_1A_2^2B_3C_0 C_1}{3L_0 L_1 } M^{-2+\frac{4}{p_0}} N^{-1+\frac{4\A}{p_0}}.
\end{equation*}
The estimate of
$\|\mathscr{D}_{ (k-1)\delta_N } (W,V,W)\|_{\tilde{X}^{p_0}_{p'_0}(I_k)}$ is similar.  Finally, 
by (\ref{NL52}) and (\ref{J3est2})  we have
\begin{align*}
\|\mathscr{D}_{ (k-1)\delta_N } (W,W,W)\|_{\tilde{X}^{p_0}_{p'_0}(I_k)} &\le 
C_1 k^{ 1-\frac{2}{p_0}} \delta_N^{ 1-\frac{2}{p_0}} \|W^{(k)} 
\|_{X_{1}^{p_0}(I_k)}^3 \\
&\le C_1 k^{1-\frac{2}{p_0}} \times B_4 M^{-4+\frac{8}{p_0}} N^{-4\A+\frac{8\A}{p_0}} 
\times A_2 N^{-3}\\
&\le A_2^3 B_4 C_1 k^{ 1-\frac{2}{p_0}} M^{-4+\frac{8}{p_0}}
N^{-4\A+\frac{8\A}{p_0}-3  } \\
&\le A_2^3 B_4 C_1 M^{-4+\frac{8}{p_0}}
N^{-4\A+\frac{8\A}{p_0}-3  } \times \frac{C_0}{3L_0 L_1}N^{4\A-\frac{4\A}{p_0}+2} \\
&\le \frac{A_2^3 B_4 C_0 C_1  }{3L_0 L_1}  M^{-2+\frac{4}{p_0}} N^{-1+\frac{4\A}{p_0}} .
\end{align*}
Therefore we get
\begin{align*}
    \|\Phi^{(k)} w^{(k)} \|_{Y^{p_0}_{p_0'} (I_k)} 
    \le 2\left(
    2A_1^2A_2B_2C_1 +\frac{2A_1A_2^2B_3C_0 C_1  }{3L_0 L_1} +\frac{A_2^3B_4C_0C_1}{3L_0 L_1} \right) M^{-2+\frac{4}{p_0}} N^{-1+\frac{4\A}{p_0}},
\end{align*}
which implies $\Phi^{(k)} :\mathscr{V}^{(k)} \to\mathscr{V}^{(k)}$ is well defined if one takes $M>1$ large enough so that
\begin{equation}
    2 \left(
    2A_1^2A_2B_2 C_1 +\frac{2A_1A_2^2B_3C_0 C_1  }{3L_0 L_1 } +\frac{A_2^3B_4C_0C_1}{3L_0 L_1} \right) M^{-2+\frac{4}{p_0}}<1. \label{mcond3}
\end{equation}
In much the same manner, we can show that
$\Phi^{(k)}$ is a contraction mapping.  For $w_1,w_2 \in \mathscr{V}^{(k)}$ we write $\Phi^{(k)}w_1-\Phi^{(k)}w_2$ in terms of $\mathscr{D}_{(k-1)\delta_N}$ as in Step 1.  We first have
\begin{equation*}
    \|\mathscr{D}_{ (k-1)\delta_N } (V^{(k)},V^{(k)},W_1^{(k)}-W_2^{(k)})\|_{\tilde{X}^{p_0}_{p'_0}(I_k)} \le 
    C_1 A_1^2  \tilde{B}_0 M^{-2}  \| W_1^{(k)}-W_2^{(k)} \|_{X_{p_0'}^{p_0}(I_k)}.
\end{equation*}
The norm $\|\mathscr{D}_{ (k-1)\delta_N } (V,W_1-W_2,V^{(k)})\|_{\tilde{X}^{p_0}_{p'_0}}$ is also smaller than the right hand side of the above inequality.  For the norms $\|\mathscr{D}(W_i^{(k)},V^{(k)},W_1^{(k)}-W^{(k)}_2)\|_{}$,\,
$\|\mathscr{D}(V^{(k)},W^{(k)}_1-W^{(k)}_2,W^{(k)}_i)\|_{}$,\, $\|\mathscr{D}(V^{(k)},W^{(k)}_i,W^{(k)}_1-W^{(k)}_2)\|_{ },\,i=1,2$, these are smaller than
\begin{align*}
C_1 (k\delta_N)^{\frac{1}{2}-\frac{1}{p_0}}& \delta_N^{\frac{1}{4}-\frac{1}{2p_0}} \|V^{(k)} \|_{ \tilde{X}_{1}^{2}(I_k)} \|W_i^{(k)} \|_{\tilde{X}_{1}^{p_0}(I_k) } \| W_1-W_2 \|_{\tilde{X}_{1}^{p_0}(I_k)} \\
&\le k^{\frac{1}{2}-\frac{1}{p_0}} \delta_N^{\frac{3}{4}-\frac{1}{2p_0}}
\|V^{(k)} \|_{ \tilde{X}_{1}^{2}(I_k)} \|W_i^{(k)} \|_{\tilde{X}_{1}^{p_0}(I_k) } \| W_1-W_2 \|_{\tilde{X}_{p'_0}^{p_0}(I_k)}\\
&\le A_1 A_2  \tilde{B}_1 C_1 M^{-3+\frac{2}{p_0}} N^{-2\A+\frac{2\A}{p_0}-1} k^{\frac{1}{2}-\frac{1}{p_0}} \| W_1-W_2 \|_{\tilde{X}_{p'_0}^{p_0}(I_k)}\\
&\le \frac{A_1 A_2  \tilde{B}_1 C_0 C_1 }{3L_0 L_1} M^{-2} \| W_1-W_2 \|_{\tilde{X}_{p'_0}^{p_0}(I_k)},
\end{align*}
where we used (\ref{J2est1}) in the last inequality.
Finally, we have
\begin{align*}
\|\mathscr{D}_{ (k-1)\delta_N } (W^{(k)}_i,W^{(k)}_j,W_1^{(k)}-W^{(k)}_2)\|_{\tilde{X}^{p_0}_{p'_0}(I_k)} &\le C_1 (k\delta_N)^{1-\frac{2}{p_0}} \|W^{(k)}_i \|_{\tilde{X}_{1}^{p_0}(I_k)  }
\|W^{(k)}_j \|_{ \tilde{X}_{1}^{p_0}(I_k) } \\
 &\times \|W^{(k)}_1-W^{(k)}_2 \|_{ \tilde{X}_{1}^{p_0}(I_k) }\\
 \le C_1  k^{1-\frac{2}{p_0}} \delta_N^{1-\frac{1}{p_0}} 
 \|W^{(k)}_i \|_{\tilde{X}_{1}^{p_0}(I_k)  } &
\|W^{(k)}_j \|_{ \tilde{X}_{1}^{p_0}(I_k) } 
\|W^{(k)}_1-W^{(k)}_2 \|_{ \tilde{X}_{p_0'}^{p_0}(I_k) }\\
\le A_2^2 C_1 \tilde{B}_2 k^{1-\frac{2}{p_0}}  M^{-4+\frac{4}{p_0}}  N^{-4\A+\frac{4\A}{p_0}-2}   & \|W^{(k)}_1-W^{(k)}_2 \|_{ \tilde{X}_{p_0'}^{p_0}(I_k) } \\
\le  \frac{A^2_2 \tilde{B}_2C_0C_1 }{3L_0 L_1}  M^{-2} 
&\|W^{(k)}_1-W^{(k)}_2 \|_{ \tilde{X}_{p_0'}^{p_0}(I_k) },\\ 
\end{align*}
where (\ref{J3est2}) was used in the last inequality.  $\|\mathscr{D}_{ (k-1)\delta_N } (W^{(k)}_i,W_1^{(k)}-W^{(k)}_2,W^{(k)}_j)\|_{\tilde{X}^{p_0}_{p'_0}(I_k)} $ can also be estimated in exactly the same way.

We gather these estimates to obtain
\begin{equation*}
    \|\Phi^{(k)}w_1 -\Phi^{(k)}w_2 \|_{\tilde{Y}_{p'_0}^{p_0} (I_k) }
    \le \left( 3A_1^2 \tilde{B}_0 C_1
    +\frac{2A_1A_2\tilde{B}_1 C_0 C_1 }{L_0 L_1}
    +\frac{A_2^2\tilde{B}_2 C_0C_1}{L_0 L_1}\right)M^{-2} \|w_1-w_2\|_{\tilde{Y}_{p'_0}^{p_0} (I_k) }.
\end{equation*}
This implies that $\Phi^{(k)}$ is a contraction mapping if we $M$ is taken 
so that
\begin{equation}
   \left( 3A_1^2 \tilde{B}_0C_1 
    +\frac{2A_1A_2\tilde{B}_1 C_0 C_1 }{L_0 L_1}
    +\frac{A_2^2\tilde{B}_2 C_0C_1}{L_0 L_1}\right)M^{-2}<\frac{1}{2}. \label{mcond4}
\end{equation}
Consequently, the map $\Phi^{(k)} :\mathscr{V}^{(k)} \to \mathscr{V}^{(k)}$ is well-defined and is a contraction mapping if we assume (\ref{mcond1}), (\ref{mcond2}), (\ref{mcond3}), (\ref{mcond4}).

Therefore, the local solution in Step 1 can be extended to the time
\begin{equation}
T_N\triangleq k_0\delta_N =cN^{-4\A} N^{\frac{2p_0}{3p_0-2} (2\A+1-\frac{2\A}{p_0})}
    =cN^{\frac{-8p_0\A+4\A+2p }{3p_0-2} },
    \end{equation}
which concludes Step 2.

\end{proof}

\begin{rem}
If needed, we go back to Step 1 and choose $M>1$ so that it satisfies (\ref{mcond3})-(\ref{mcond4}) in addition to (\ref{mcond1})--(\ref{mcond2}).  Then, at each step, the solution can be extended by the same time length $\delta_N$.
\end{rem}

\textbf{Step 3}. We have shown that, for any sufficiently large $N>1$, there is a local solution $u \in Y^2_{2}(T_N) +Y^{p_0}_{p_0'}(T_N)$ of (\ref{NLS}) on the time interval $[0,T_N]$.  Note that the solution $u$ is defined for each $N$.  So we denote it by $u^{(N)}$ here.  Then, observe that $T_N \to \infty$ as $N \to \infty$ if 
\begin{equation*}
    -8p_0\A+4\A+2p_0>0
    \end{equation*}
    which is equivalent to (\ref{globalcond}).  Therefore, in this case, we see that, for any large $T>0$, there is $N(T)>0$ such that the local solution $u^{(N(T))}$ whose life span reaches $T$, and one may expect that a global solution of (\ref{NLS}) can be constructed using these local solutions.  To ensure the global existence we need the following uniqueness result:
\begin{prop} \label{uniqueness} 
Let $N_1,N_2>1$ be sufficiently large.  Then $u^{(N_1)}(t)=u^{(N_2)}(t)$ 
for any $t \in [0,\min (T_{N_1}, T_{N_2})]$.
\end{prop}

\begin{proof}
    Let $T>0$.  It is enough to show that $u_1(t)=u_2(t),\,\forall t \in[0,T]$ for two solutions $u_1$ and $u_2$ in $Y^2_{2}(T)+Y^{p_0}_{p_0'}(T)$ with $u_1(0)=u_2(0)$.  By Corollary \ref{Strinclusion} we have
    \begin{equation*}
        u_j \in L^{Q_2(r)}([0,T] ; L^r) +L^{Q_{p_0}(r)}([0,T] ;L^r),
    \end{equation*}
    for $r$ satisfying (\ref{StrregularityScaling}).  Now it is easy to see that $r=6$ satisfies (\ref{StrregularityScaling}) for any $p_0\in (2,3)$.  Thus, noting 
    that $Q_3(6)=4$ and the fact that $Q_{p_1} (r) \le Q_{p_2} (r)$ for $p_1 \ge p_2$, 
    we have
    \begin{equation*}
        u_j \in L^{Q_2(6)}([0,T] ; L^6) +L^{Q_{p_0}(6)}([0,T] ;L^6)
        \subset L^4([0,T] ; L^6).
    \end{equation*}
Thus it is enough to show the uniqueness in $L^4([0,T] ; L^6)$. 

We estimate the difference of the integral equations
\begin{equation*}
    u_j(t)=U(t)u_j (0) +i \int^t_0 U(t-s) |u_j (s)|^2 u_j(s) ds,\quad j=1,2.
\end{equation*}
We take $T_0 \in (0,T)$ and set 
\begin{equation*}
    \eta_T \triangleq \max 
( \|u_1 \|_{L^4([0,T] ; L^6)}, \|u_2 \|_{L^4([0,T] ; L^6)}).
\end{equation*}  Writing
\begin{align*}
    |u_1|^2u_1 -|u_2|^2u_2 &= u_1^2 (\bar{u}_1-\bar{u}_2) +u_1 \bar{u}_2 (u_1-u_2)
    +|u_2|^2 (u_1-u_2),
\end{align*}
and applying the standard Strichartz estimate and H\"older's inequality, we get
\begin{align*}
\|u_1 -u_2 \|_{L^{4}([0,T_0]; L^6   )} &\le T_0^{\frac{1}{12}} \|u_1 -u_2 \|_{L^{6}([0,T_0]; L^6)} \\
&=T_0^{\frac{1}{12}} \left\| \int^t_0 U(t-\tau) \left( |u_1|^2u_1 -|u_2|^2 u_2  \right) \right\|_{L^{6}([0,T_0]; L^6   )} \\
&\le C T_0^{\frac{1}{12}} \left\||u_1|^2u_1 -|u_2|^2 u_2 \right\|_{L^{1}([0,T_0]; L^2   )}\\
&\le CT_0^{\frac{1}{12}} \biggl( \|u_1^2 (\bar{u}_1-\bar{u}_2) \|_{L^1([0,T_0] ;L^2 )} + \| u_1 \bar{u}_2 (u_1-u_2) \|_{L^1([0,T_0] ;L^2 )} \\ 
& \qquad  + \||u_2|^2 (u_1-u_2)   \|_{L^1([0,T_0] ;L^2 )} \biggr) \\
\le &CT_0^{\frac{1}{3}}\left ( \|u_1\|^2_{L^4([0,T] ;L^6 )  } + \|u_1\|_{ L^4([0,T] ;L^6 )}
\|u_2\|_{L^4([0,T] ;L^6 ) } +\|u_2\|_{L^4([0,T] ;L^6 ) }^2\right) \\
& \qquad \times \| u_1 -u_2 \|_{  L^4([0,T_0] ;L^6 )} \\
& \le 3CT_0^{\frac{1}{3}} \eta_T^2 \| u_1 -u_2 \|_{  L^4([0,T_0] ;L^6 )}.
\end{align*}
Then if we take $T_0$ so that
\begin{equation*}
    3CT_0^{\frac{7}{12}-\frac{1}{p}} \eta_T^2<1,
    \end{equation*}
 we have $u_1(t) =u_2(t),\,\forall t\in [0,T_0]$.  Repeating a similar argument, the uniqueness assertion can be extended to $[0,T]$.

\end{proof}

Now we define $u(t) \triangleq u^{(N)} (t)$ if $t \in [0, T_N]$.  Then, by the above uniqueness theorem $u(t)$ is well defined on $[0,\infty)$ if $p_0,\A$ satisfies (\ref{globalcond}) and we
get a global solution of (\ref{NLS}).  The regularity in the Strichartz spaces follows immediately from Corollary \ref{Strinclusion}. This completes the proof of Theorem \ref{keyGWP}.

\subsection{Proof of Theorem \ref{Maintheorem}}  The existence part of Theorem \ref{Maintheorem} is obtained as a corollary of Theorem \ref{keyGWP}.  This is due to the following proposition:
\begin{prop} \label{LPinclude}
Let $p_0>2$ and $\theta \in (0,1)$.  
Let $p_{\theta}$ be given by
\begin{equation}
    \frac{1}{p_{\theta}}=\frac{1-\theta}{p_0}+\frac{\theta}{2} \label{LPinterpolation}
\end{equation}
for some $\theta \in (0,1)$.
Then,
\begin{equation*}
L^{p_{\theta}} \subset D_{p_0,\frac{1-\theta}{\theta}}.
\end{equation*}

\end{prop}

To prove the proposition we need the following lemma:

\begin{lem} {\rm (See e.g. \cite{HTT, Triebel})}  \label{interpolation}
Let $1\le q_1<q_2 <\infty$ and let
\begin{equation*}
    \frac{1}{q}=\frac{1-\theta}{q_1}
    +\frac{\theta}{q_2}
\end{equation*}
for some $\theta \in (0,1)$.  Then
$L^q \subset L^{q_1}+L^{q_2}$ and there 
are sequences $(f^t_1)_{t>0} \subset  L^{q_1}$ and $(f^t_2)_{t>0} \subset L^{q_2}$ such that $f=f^t_1+f^t_2$ 
    \begin{equation*}
        ct^{-\theta} \max \left(\|f_1^t\|_{L^{q_1}} ,
        t\|f_2^t\|_{L^{q_1}} 
        \right) \le \|f\|_{L^q}
    \end{equation*}
    for any $t>0$.
\end{lem}

\noindent\textit{Proof of Proposition \ref{LPinclude}}.\quad
It is clear that $L^2 \subset D_{p_0,\frac{1-\theta}{\theta}}$.  Thus, we let $\phi \in L^p \setminus L^2$ and we show that $\phi \in D_{p_0,\frac{1-\theta}{\theta}}$.  By Lemma \ref{interpolation} there are $(\phi_t^{(2)})_{t>0} \subset L^2$ and 
$(\phi_t^{(p_0)} )_{t>0} \subset L^{p_0}$ such that $\phi=\phi_t^{(p_0)}+\phi_t^{(2)}$ and
\begin{equation}
    ct^{-\theta} \max \left( 
    \|\phi_t^{(p_0)} \|_{L^{p_0}} , t\|\phi_t^{(2)}\|_{L^2} \right) \le 
    \|\phi \|_{L^p}.\label{interpolationineq}
\end{equation}
Note that $\| \phi_t^{(2)} \|_{L^2} \to \infty$ as $t\to \infty$ since $\phi \notin L^2$.  Thus for $N>0$ there exists $t_N>0$ such that
$\|\phi_{t_N}^{(2)} \|_{L^2} =N^{\frac{1-\theta}{\theta}}$.  Now we set
\begin{equation*}
    \varphi_N \triangleq \phi_{t_N}^{(2)},\quad
    \psi_N \triangleq \phi_{t_N}^{(p_0)}.
\end{equation*}
    Then by (\ref{interpolationineq}) we have
    \begin{equation*}
        t_N^{1-\theta} \le c^{-1} \|\phi\|_{L^p} N^{-\frac{1-\theta}{\theta}},\qquad 
        \|\psi_N \|_{L^{p_0}} \le c^{-1} t_N^{\theta} \|\phi\|_{L^p}.
    \end{equation*}
    These estimates yield
    \begin{equation*}
        \| \psi_N \|_{L^{p_0}} 
        \le (c^{-1} \|\phi \|_{L^p} )^{\frac{1}{1-\theta}} N^{-1}.
    \end{equation*}
    This implies $\phi \in D_{p_0,\A}$.
\qed

Now we conclude the proof of the global existence part.  Let $\epsilon>0$ be sufficiently small.  We put $p_0=p_0(\epsilon) \triangleq 3-\epsilon$ and 
for $\theta \in (0,1)$ we define $p_{\theta}
=p_{\theta}(\epsilon)$ by
\begin{equation*}
    \frac{1}{p_{\theta}} \triangleq\frac{1-\theta}{p_0} +\frac{\theta}{2}.
\end{equation*}
Then, by Propositon \ref{LPinclude} we have
\begin{equation*}
    L^{p_{\theta}}  \subset D_{p_0,\frac{1-\theta}{\theta}}.
\end{equation*}
Thus, by Theorem \ref{keyGWP} we see that, if
\begin{equation}
    \frac{1-\theta}{\theta} <\frac{p_0}{2(2p_0-1)},
\end{equation}
there is a global solution of (\ref{NLS}) for any
$\phi \in L^{p_{\theta}}$.  Now, we define $\theta_0=\theta_0(\epsilon)$ by
\begin{equation}
    \frac{1-\theta_0}{\theta_0} =\frac{p_0}{2(2p_0-1)}.
\end{equation}
Then, the sufficient condition for the global existence for $\phi \in L^p$ is rephrased as 
$2<p<p_{\theta_0}$.  Now, a computation shows that 
\begin{equation*}
    p_{\theta_0} =\frac{13-5\epsilon}{6-2\epsilon}.
    \end{equation*}
This implies a global solution $u$ such that
\begin{equation}
    u|_{[0,T]} \in Y_2^2(T)+ Y^{p_0}_{p_0'}(T),\quad \forall T>0 \label{Yregularity} 
    \end{equation}
    exists for any $\phi \in L^p$ if $2\le p <13/6$, since in this case we may have $p<p_{\theta_0}(\epsilon)$ for a sufficiently small $\epsilon>0$. (\ref{Strregularity}) is an immediate consequence of (\ref{Yregularity}) and Corollary \ref{Strinclusion}, since $p_0=3-\epsilon$ and $r$ satisfy (\ref{StrregularityScaling}) if $3 \le r \le 6$.

It remains to prove the persistence property of
the twisted variable $v(t)=U(-t) u(t)$.  We exploit generalized Strichartz estimates for $L^p$-data for $p<2$.

\begin{prop}{\rm (\cite[Theorem 3.2]{Kato})}\label{Lpstrprop}
Let $1<\rho\le 2$ and let $2\le \gamma,\sigma < \infty$ be such that $1/\gamma+1/\sigma<1/2$ and
\begin{equation*}
    \frac{2}{\gamma}+\frac{1}{\sigma}=\frac{1}{\rho}.
\end{equation*}
Then
\begin{equation}
    \|U(t) f\|_{L^{\gamma}(\R ;L^{\sigma})} \le 
    C\| f \|_{L^{\rho}}.\label{Lpstr}
\end{equation}
\end{prop}

Here we use the following equivalent form of
this estimate:
\begin{cor}
Let $J\subset \R$ be an interval.  Suppose that $\gamma,\sigma,\rho$ satisfy the 
assumptions of Proposition \ref{Lpstrprop}.  Then
\begin{equation}
    \left\| \int_I U(-s) F(s) ds\right\|_{L^{\rho'}}
    \le C \|F\|_{L^{\gamma'}(J ;L^{\sigma'})}
\end{equation}
for all intervals $I\subset J$, where $C$ is independent of $J$.
\end{cor}

\begin{proof}
Although the corollary follows immediately
from the standard duality argument, we prove it
for the sake of completeness.  Let $I \subset J$. We have
\begin{align*}
    \left\|\int_I U(-s) F(s)  ds\right\|_{L^{\rho'}} &=
    \sup_{f\in L^{\rho},\,\|f\|_{L^{\rho}}=1} 
    \int_{\R} \int_I [U(-s) F(s)](x) ds \overline{f(x)} dx \\
    &=
    \sup_{f\in L^{\rho},\,\|f\|_{L^{\rho}}=1} 
    \int_{\R} \int_I \int_{\R} 
    \int_{\R} e^{is\xi^2 +ix\xi-iy\xi} F(s,y)\overline{f(x)}  dy  d\xi ds  dx \\
    &=
    \sup_{f\in L^{\rho},\,\|f\|_{L^{\rho}}=1} 
    \int_{I}  \int_{\R} F(s,y) 
    \overline{ \int_{\R} e^{iy\xi-is\xi^2} \int_{\R}
    e^{-ix\xi} f(x)  dx d\xi} dyds \\
    &= \sup_{f\in L^{\rho},\,\|f\|_{L^{\rho}}=1}
    \int_{I}  \int_{\R} F(s,y) \overline{ [U(s)f](y)} dy.
\end{align*}
Now, by (\ref{Lpstr}) and H\"older's inequailty, the right hand side is smaller than
\begin{align*}
\sup_{f\in L^{\rho},\,\|f\|_{L^{\rho}}=1}
\|F \|_{L^{\gamma'}(I ; L^{\sigma'})} \| U(s)f \|_{L^{\gamma} (I ;L^{\sigma})}  &\le C 
\sup_{f\in L^{\rho},\,\|f\|_{L^{\rho}}=1}
\|F \|_{L^{\gamma'}(I ; L^{\sigma'})} \|f \|_{L^{\rho}} \\
&\le C\|F \|_{L^{\gamma'}(J; L^{\sigma'})} .
\end{align*}

\end{proof}

Now we are ready to prove the persistence property.  We start from the integral equation
equivalent to (\ref{NLS}).  Clearly, we get
\begin{equation*}
    U(-t)u(t)=\phi +i\int^t_0 U(-s) |u(s)|^2 u(s) ds.
\end{equation*}
    Then it is enough to estimate the $L^p$ norm of Duhamel part in the right hand side.  
Let $T>0$.  Observe first that we have

$u \in L^{Q_{3-\epsilon}(4)}([0,T] ;L^4)\subset L^{Q_{3}(4)}([0,T] ;L^4) 
=L^{\frac{24}{5}}([0,T] ; L^4)$.

We apply Proposition \ref{Lpstrprop} with $\gamma=\frac{8p}{3p-4},\,\sigma=4,\,\rho=p'$ which satisfies the assumption if $p<4$.  For $t \in [0,T]$ we obtain
\begin{align}
\left\|\int^t_0 U(-s) |u(s)|^2 u(s) ds   \right\|_{L^p} &\le C \left\| |u|^2 u  \right\|_{ L^{\frac{8p}{5p+4}  } ( [0,T] ; L^{\frac{4}{3}})} \\
& \le C\|u\|^3_{ L^{\frac{24p}{5p+4}  } ( [0,T] ; L^4)} \\
&\le CT^{\frac{1}{6p}} \| u\|^3_{L^{\frac{24}{5}}( [0,T] ; L^4)}.
\end{align}
Thus we see that $U(-t)u(t) \in L^p$ for any $t \in T$.  Since $T$ can be taken arbitrarily large, we have $U(-t)u(t) \in L^p$ for every $t>0$.  


\begin{thebibliography}{99}
\bibitem{Bourgain} J. Bourgain, {\em  Refinements of Strichartz's inequality and
 applications to 2D-NLS with critical
 nonlinearity}. Int. Math. Res. Not., No. 5, (1998) 253-283.
\bibitem{Brenner} P. Brenner, V. Thom\'ee and L.B. Wahlbin, {\em Besov Spaces and Applications to Difference Methods for Initial Value Problems}, Lecture Notes in Math. Springer 434.

\bibitem{Caz} T. Cazenave, {\em Semilinear  Schr\"odinger equations}, Courant Lect. 
Notes Math. {\bf 10}, New York Univ., Courant Inst. Math. Sci., New York, 2003.
\bibitem{CVV} T. Cazenave, L. Vega, and M.C.Vilela, {\em A note on the nonlinear Schr\"odinger equation in weak $L^p$ spaces}, Communications in contemporary Mathematics, Vol. 3, No.1 (2001),153--162.

\bibitem{NLSModu} L. Chaichenets, D. Hundertmark, P. Kunstmann, and
N. Pattakos, {\em On the existence of global solutions of the one-dimensional cubic NLS for initial data in teh modulation space $M_{p,q}(\R)
$}, J. Differential Equations {\bf 263} (2017) 4429--4441.

\bibitem{DSS} B. Dodson, A. Soffer, and T. Spencer, {\em Global well-posedness for the cubic nonlinear Schr\"odinger equation with initial data lying in $L^p$-based Sobolev spaces}, J. Math. Phys. {\bf 62} (2021), 071507.
\bibitem{Fefferman} C. Fefferman, {\em Inequalities for strongly singular convolution operators,} Acta math. 124 (1970), 9-36.


\bibitem{GrunrockKdV} A. Gr\"unrock, {\em An improved local well-posedness result for the modified KdV equation}, Int. Math. Res. Not,. No. 61, (2004), 3287-3308.
\bibitem{Grunrock} A. Gr\"unrock, {\em Bi- and trilinear Schr\"odinger estimates in one space dimension with applications to cubic NLS and DNLS  },
Int. Math. Res. Not., No. 41, (2005), 2525-2558.

\bibitem{HN} N. Hayashi and P. Naumkin, {\em Asymptotics for large time of solutions to the nonlinear Schr\"odinger and Hartree equations}, Amer. J. Math. {\bf 120} (1998), 369--389.
\bibitem{Hormander} L. H\"ormander, {\em Estimates for translation invariant operators in $L^p$ spaces,} Acta Math, {\bf 104} (1960), 141--164.


\bibitem{107indiana} R. Hyakuna, {\em Global solutions to the Hartree equation for large $L^{p}$-initial data}, Indiana Univ. Math. J. {\bf 68} (2019), 1149--1172.

\bibitem{107jfa} R. Hyakuna, {\em Local and global well-posedness, and $L^{p'}$-decay estimates for 1D nonlinear Schr\"odinger equations with Cauchy data in $L^p$}, J. Funct. Anal. {\bf 278}, No. 12, (2020), 108511.
\bibitem{107slow} R. Hyakuna, {\em Well-posedness for the 1D nonlinear Shcr\"odinger equatin in $L^p,p>2$}, Nonlinear Anal. T.M.A {\bf 226}, (2023), 113154.

\bibitem{HTT} R. Hyakuna, T. Tanaka, and M. Tsutsumi, {\em On the global well-posedness for the nonlinear Schr\"odinger equations with large initial data of inifinite $L^2$ norm}, Nonlinear Anal. {\bf 74} (2011) 1304--1319.

\bibitem{107T} R. Hyakuna and M. Tsutsumi, {\em On existence of global solutions of Schr\"odinger equations with subcritical nonlinearity for $\widehat{L^p}$-data}, Proc. Am, Math. Soc. {\bf 140} (2012), 3905--3920.

\bibitem{scjfa} R. Schippa, {\em On smoothing estimates in modulation spaces and the nonlinear Schr\"odinger
equation with slowly decaying initial data}, J. Funct. Anal. {\bf 282 (5)} (2022), 109352.


\bibitem{Kato} T.Kato, {\em An $L^{q,r}$-theory for nonlinear Schr\"odinger equations}, Adv. Stud. Pure Math., vol. 23, Math. Soc. Japan, Tokyo, 1994, 223-238.

\bibitem{Triebel} H. Triebel, {\em Theory of Function Spaces,} Birkh\"auser, Basel 1983.

\bibitem{VV} A. Vargas and L. Vega, {\em Global wellposedness
	for 1D nonlinear Schr\"odinger equation for data with an
	infinite $L^2$ norm}, J. Math. pures Appl. {\bf 80}, (2001),
	1029-1044.
\bibitem{yT} Y. Tsutsumi, {\em $L^2$ -solutions for nonlinear
	Schr\"odinger equations and nonlinear groups}, Funkcial Ekvac., 
        {\bf 30} (1987), 115-125.


\bibitem{Zhou} Y. Zhou, {\em  Cauchy problem of nonlinear Schr\"odinger equation with initial data in Sobolev spce $W^{s,p}$ for $p<2$                     }, 
Trans. Amer. Math. Soc., {\bf 362} (2010), 4683-4694.
\end{thebibliography}
\end{document}